\documentclass[twoside,reqno]{amsart}

\usepackage{amsmath}
\usepackage{amssymb,latexsym,amstext,amsthm}
\usepackage{verbatim} 
\usepackage{eucal} 
\usepackage{color}
\theoremstyle{plain}
\newtheorem{thm}[equation]{Theorem}
\newtheorem{pro}[equation]{Proposition}
\newtheorem{cor}[equation]{Corollary}
\newtheorem{lem}[equation]{Lemma}
\theoremstyle{definition}
\newtheorem{ex}[equation]{Example}

\newtheorem{DEF}[equation]{Definition}
\theoremstyle{remark}
\newtheorem{rem}[equation]{Remark}

\newcommand{\sub}{\subseteq}

\newcommand{\la}{Lie algebra }
\newcommand{\lfrs}{locally finite root system }

\newcommand{\lra}{\longrightarrow}
\newcommand{\ve}{\varepsilon}

\newcommand{\F}{\mathbb{F} }

\newcommand{\Z}{\mathbb{Z} }

\newcommand{\CA}{\mathcal{A} }

\newcommand{\CB}{\mathcal{B}}

\newcommand{\CK}{\mathcal{K}}
\newcommand{\CV}{\mathcal{V}}
\newcommand{\fg}{\mathfrak{g}}

\newcommand{\pa}{{\pi(\alpha)}}
\newcommand{\ep}{\hfill$\Box$}

\def\form{(\cdot,\cdot)}
\def\ad{\hbox{ad}}
\def\andd{\quad\hbox{and}\quad}
\def\sg{\sigma}
\def\a{\alpha}
\def\b{\beta}
\def\lam{\lambda}
\def\Lam{\Lambda}
\def\ep{\epsilon}
\def\andd{\quad\hbox{and}\quad}
\def\supp{\hbox{supp}}
\def\id{\hbox{id}}
\def\Aut{\hbox{Aut}}

\def\rred{\textcolor{black}}

\begin{document}


\setcounter{page}{1} \setcounter{page}{1}

\author[Azam, Hosseini, Yousofzadeh]{Saeid Azam$^{1}$, S. Reza Hosseini$^{1}$, Malihe Yousofzadeh$^{2}$}

\title[Extended affinization of IARA's]{Extended affinization of Invariant Affine Reflection Algebras}


\address
{
School of Mathematics, Institute for Research in Fundamental Sciences (IPM), P.O. Box: 19395-5746,
Tehran, Iran, and Department of Mathematics\\ University of Isfahan\\Isfahan, Iran,
P.O.Box 81745-163.} \email{azam@sci.ui.ac.ir, saeidazam@yahoo.com.}

\address
{Department of Mathematics\\ University of Isfahan\\Isfahan, Iran\\
P.O.Box 81745-163} \email{srhosseini@sci.ui.ac.ir, srh\_umir@yahoo.com.}

\address{Department of Mathematics\\ University of Isfahan\\Isfahan, Iran,
P.O.Box 81745-163 and School of Mathematics, Institute for
Research in Fundamental Sciences (IPM), P.O. Box: 19395-5746,
Tehran, Iran.}\email{ma.yousofzadeh@sci.ui.ac.ir.}

\thanks{$\;^1$This research was in part supported by a grant from IPM
(No. 90170217)}

\thanks{$\;^2$This research was in part supported by a grant from IPM
(No. 90170031)}

\thanks{
The authors would like to thank the Center of Excellence for
Mathematics, University of Isfahan}


\begin{abstract}
  The class of invariant affine reflection algebras is the most general known extension of the class
 of affine Kac-Moody Lie algebras, introduced
  in 2008. We develop a method known as ``affinization'' for the class of invariant affine reflection algebras, and
  show that starting from an algebra belonging to this class together with a certain finite order automorphism, and applying the
  so called ``affinization method'', we obtain again an invariant affine reflection algebra.
 This can be considered as an important step towards the  realization of invariant affine reflection algebras.
\end{abstract}
\maketitle
\vspace{5mm} \setcounter{section}{-1}
\section{\bf Introduction}\label{introduction}

The class of affine Kac-Moody Lie algebras has been of great interest in the past fifty years, mostly for its applications
to various areas of Mathematics and Theoretical Physics. This has been a strong motivation for mathematicians to extend this class.
Among such extensions, the most important ones are
the class of {\it extended affine Lie algebras} \cite{aabgp},
the class of  {\it toral type extended affine Lie algebras} \cite{aky,you},
the class of {\it locally extended affine Lie algebras} \cite{myleala} and the most recent one which covers
 all of the previous ones, the class of  {\it invariant affine reflection algebras} (IARA's for short), introduced in 2008
 by E.  Neher \cite{nehersurvey}.

One of the central concepts of the theory of affine Kac-Moody Lie algebras and its extensions,
which has captured the interest of many mathematicians, is the concept of ``realization''.
Historically, the most popular way of realizing affine Lie algebras and their generalizations is
a developed version of a method known as ``{\it affinization}'', due to V. Kac \cite[Chapter 8]{kac}.
Roughly  speaking, the method of affinization can be described as follows. Let $\fg$ be a Lie algebra from a class $\mathcal T$,
$\CA$ the ring of Laurent polynomials, and $\sg$ a finite order automorphism of $\fg$. Then applying the
 affinization method to these data, one obtains another element $\hat\fg=\tilde{\fg}\oplus C\oplus D$ of the class $\mathcal T$, where $\tilde{\fg}$
 is a subalgebra of the loop algebra $\fg\otimes\CA$, $C$ is a subspace contained in the center and $D$ consists of certain derivations.

One knows that affine Kac-Moody Lie algebras, which are extended affine Lie algebras of nullity one (see \cite{abgp}),
are obtained through the method
of affinization starting form finite dimensional simple Lie algebras, which are extended affine Lie algebras of nullity zero.
It is therefore natural to ask ``whether  it is possible to obtain (to realize) extended affine
 Lie algebras of higher nullity from the ones with lower nullity, through the affinization method.
This question was positively answered by U. Pollmann \cite{POL}, where she realized extended affine Lie algebras of nullity $2$,
up to derivations and central {extensions}, starting from the ones of nullity one. In the past two decades, there {have} been
several other attempts of applying the affinization method, either directly or indirectly by using a closely related method,
in order to realize extended affine Lie algebras; see for example \cite{atwist,aby,ak,youfixedpoint}.
 In \cite{abpcovering1}, the method of affinization was defined in a general setting, in fact
this setting provides a framework of producing new Lie algebras from the old ones in a prescribed way.
 The mentioned work was led to realization of almost all centerless Lie tori (see \cite{abpiterated,abfpgs,abfpeala}),
 a class of Lie algebras characterizing  the core modulo center
 of extended affine Lie algebras.

In this paper, we consider the method of affinization for the class of IARA's, in an extended way. Namely,
in our method, the ring of Laurent polynomials is replaced with a certain associative algebra, and moreover, the way of inserting the central elements and
derivations to the construction allows us to produce IARA's of arbitrary higher nullity from the ones we start with.
So our work extends the results of \cite{abpcovering1}, and in part \cite{ay}.

The paper is organized as follows. In Section 1, we gather preliminary definitions and results needed throughout the work.
 In Section 2, we study two special types of gradings imposed by certain automorphisms on the underlying  Lie algebras.
 In Sections 3 and 4, we study the effect of these gradings on so-called toral pairs in general and on IARA's in particular.
 In the latter case, it is shown that if the corresponding toral subalgebra is replaced with its degree zero homogeneous subspace,
one gets a new IARA with a generally different root system.
In Section 5, as a by-product of  the results in earlier sections, we show that the fixed point subalgebra of an IARA
under a certain finite order automorphism is again an IARA. This gives a new perspective to an old question, going back to \cite{bm},
concerning the structure of fixed point
subalgebras. Finally, Sections 6 and 7 are devoted to our results on affinization of IARA's. Roughly
speaking, we show that the outcome of ``affinization'' of an IARA under a certain automorphism is again an IARA.
We consider this as an important step towards realization of IARA's.
We use our method to give examples of IARA's which are neither locally extended affine Lie algebras nor toral type extended affine Lie algebras.

The authors would like to thank Professor Eerhard Neher and Professor Mohammad-Reza Shahriary for some helpful comments on the early version of this work.

\section{\bf Preliminaries}\setcounter{equation}{0}\label{pre}
     In this section{,} we gather preliminary definitions and results which we need  throughout the paper.
In this work,
   all vector spaces are considered over a field $\F$ of characteristic zero.
  For any vector space $W,$ we denote its dual space  by $W^\star.$ For a nonempty set $S,$ by $id_S,$ we mean the identity map on
$S$ and by $|S|$ the cardinal number of $S.$  If  $R$ is   an
integral domain with the field of fractions $Q,$ $A$  an
$R$-module and $S$ a subset of $A,$ we denote by $\langle
S\rangle,$ the $R$-span of $S.$  A map $\form: A\times
A\longrightarrow Q$ is called a {\it symmetric bihomomorphism} if
$\form $ is an  $R$-module homomorphism on each component and
$(a,b)=(b,a)$ for all $a,b\in A.$  For a symmetric
 bihomomorphism
 $\form: A\times A\rightarrow Q$, the set $A^0:=\{a\in A\mid (a,b)=0;\hbox{ for all }b\in A\}$ is called the {\it radical} of the $\form.$ We also set
$$S^0:=S\cap A^0 \quad\hbox{and}\quad S^\times:=S\setminus S^0.
$$

The elements of $S^0$ (resp. $S^\times$) are called {\it
isotropic} (resp. {\it nonisotropic}) elements of $S$. A subset
$S$ of $ A$ is called {\it indecomposable} or {\it connected} if
$S^\times$ cannot be written as a disjoint union of two its
nonempty orthogonal subsets with respect to $\form$. In the
special case when $R=\Z,$ the bihomomorphism $\form$ is called a
{\it  positive definite form} (resp. {\it positive semidefinite
form}) if $(a,a)>0$ (resp. $(a,a)\geq0$) for all nonzero $a\in A.$
For a subset $S$ of $A$ equipped with a positive semidefinite form
$\form$, we have
$$S^0=\{\alpha\in S\mid (\alpha,\alpha)=0\}\andd S^\times=\{\alpha\in S\mid (\alpha,\alpha)\not=0\}.
$$

\begin{DEF} \label{toralsubalg}
  Let $\fg$ be a \la and $T \subseteq \fg$ a subalgebra, we call $T$ a
  {\it toral} subalgebra  or an {\it ad-diagonalizable} subalgebra if
\begin{equation}\label{rootdeco}
  \fg = \bigoplus_{\alpha \in T^\star} \fg_\alpha(T) \quad
\end{equation}
  where for any $\alpha \in T^\star $,
  \[\fg_\alpha(T) :=  \left\{ x \in \fg \mid [t,x] = \alpha(t) x, \hbox{for all }t \in T \right\}. \]
   In this case $(\fg, T)$ is called a {\it toral pair}, the decomposition  (\ref{rootdeco})  the \emph{root
   space decomposition} of $(\fg, T)$ and $R: =
   \left\{ \alpha \in T^\star \mid \fg_\alpha(T) \neq 0 \right\}$
   the {\emph{root system}} of $(\fg, T)$. We will usually abbreviate $\fg_\alpha(T)$ by $\fg_\alpha$.
   {Since any toral subalgebra is abelian, $T \subseteq \fg_0$ and so $0 \in R$ unless $T= \{ 0 \} = \fg.$ A toral subalgebra is called a \emph{splitting Cartan subalgebra}
   if $T = \fg_0$, in this case $(\fg,T)$ is called a {\it split} toral pair}.
\end{DEF}

Now let  $(\fg,T)$ be  a toral pair with root system $R$, namely
$\fg = \bigoplus_{\alpha \in R} \fg_\alpha$. Suppose that $\fg$
satisfies the following two axioms:
  \begin{itemize}
    \item[(IA1)] $\fg$ has an invariant nondegenerate symmetric bilinear form $\form$
    whose restriction to $T$ is nondegenerate.
    \item[(IA2)] For each {$\alpha \in R\setminus\{0\}$}, there exist $e_\alpha \in \fg_\alpha$
    and $f_\alpha \in \fg_{-\alpha}$ such that $0 \neq [e_\alpha, f_\alpha] \in T$.
\end{itemize}
One can see that for each $\alpha \in R$, there exists a unique
$t_\alpha \in T$ which represents $\alpha$ via $\form$
  (i.e. $\alpha(t) = (t_\alpha, t)$ for all $t \in T$) and  that the map $\nu:T \lra T^\star$ given by $\nu(t)=(t,\cdot)$ is a monomorphism whose image
  contains $\hbox{span}_{\F}R$. 
Now it follows that the  bilinear form on  $T$ can be transferred
to a bilinear form on span$_\F R$ defined  by
$$(\alpha,\beta)=(t_\alpha,t_\beta),\hbox{ for all }\alpha, \beta \in
 \hbox{span}_\F R.$$

 Here, we record the definition of an invariant affine reflection algebra,
 the main object of this study.
\begin{DEF}\label{iara}\cite[Section 6.7]{nehersurvey}
Let $(\fg,T)$ be a toral pair with root system $R$. Assume $\fg
\neq 0$. The pair $(\fg,T)$ (or simply $\fg$) is called an
\emph{invariant affine reflection algebra} (IARA for short) if it
satisfies (IA1), (IA2) as above and (IA3) below:
\begin{itemize}
    \item[(IA3)] For every $\alpha \in R$ with $(\alpha, \alpha) \neq 0$ and for all $x_\alpha \in \fg_\alpha$,
    the adjoint map $\ad x_\alpha$ is locally nilpotent on $\fg$.
  \end{itemize}

  We call an invariant affine reflection algebra $(\fg,T)$  {\it division}, if (IA2) is replaced with the
  stronger axiom (IA2)$'$ below:
  \begin{itemize}
  \item[(IA2)$'$] For each $\alpha\in R\backslash \{0\} $ and any $0\neq e_\alpha\in \fg_\alpha,$ there
  exists $f_\alpha \in \fg_{-\alpha}$ such that $0\neq[e_\a,f_\a]\in T$.
\end{itemize}
\end{DEF}
\begin{rem}\label{temprem1} (i) In this work, we always assume for
a  toral pair $(\fg,T)$ satisfying (IA1), the corresponding root
system is  not the zero set.

(ii) If $(\fg, T)$ is a {split} toral pair, then axiom (IA1)
implies (IA2)$'$, in particular any invariant affine reflection
algebra with a splitting {Cartan} subalgebra is division. To see
this, one can combine Lemma~\ref{tpialpha} and (\ref{new-1})
below.
\end{rem}

Let us also recall the definition of an affine reflection system.
This notion is due to E. Neher \cite[Chapter 3]{nehersurvey} but
here we state  an equivalent definition given in \cite[Definition
1.3]{AYY}.

\begin{DEF}\label{ARS}
Let $A$ be an abelian group equipped with a nontrivial symmetric
positive semidefinite form $\form$ and $R$ be a subset of $A$. The triple $(A,\form,R),$ or $R$ if
there is no confusion, is called an {\it affine reflection system}
if it satisfies the following 3 axioms:

(R1) $R=-R$,

(R2) $\langle R\rangle=A$,

(R3) for $\alpha\in R^{\times}$ and $\beta\in R$, there exist $d, u\in
{\mathbb Z}_{\geq 0}$ such that
$$
(\beta+{\mathbb Z}\alpha)\cap R=\{\beta-d\alpha,\ldots,\beta+u\alpha\}\quad\hbox{and}\quad
d-u=(\beta,\alpha^\vee).
$$
Each element of $R$ is called a {\it root.} Elements of $R^\times$
(resp. $R^0$) are called  {\it non-isotropic  roots} (resp. {\it
isotropic roots}).

The affine reflection system $R$ is called {\it irreducible}{, if}

(R4) $R^\times$ is indecomposable.

\noindent Moreover, $R$ is called {\it tame}{,} if

(R5) $R^0\sub R^{\times}- R^{\times}$ (elements of $R^0$ are \emph{non-isolated}).
%
%

  {A {\it locally finite root system} is, by definition, an affine reflection system for which $A^0=\{0\}$, see \cite{nllfrs,AYY}.}
\end{DEF}
\begin{rem}\label{rem}
  It is shown in \cite{nehersurvey} that the root system  $R$ of an IARA $(\fg,T)$ is an affine reflection system in the $\Z$-span of $R$.
We note that as in this case
  $R \subseteq T^\star$ and $\F$ is of characteristic zero,the $\Z$-span of $R$ is a torsion free abelian group.
\end{rem}

\begin{lem}\label{tpialpha}
 Let $(\fg,T)$ be a toral pair, with root system $R,$ satisfying (IA1) and (IA2). If
$\alpha \in R$, $x \in \fg_{\alpha}$, $y \in \fg_{-\alpha}$ and
$[x , y] \in T$, then $[x,y]=(x,y) t_{\alpha}$.
\end{lem}
\begin{proof}
      We will show that $ [x,y] - (x,y)t_{\alpha}$ is an element of the radical of the form on $T${;}
      then we are done as $\form$ is nondegenerate on $T$. For this, suppose $t \in T$ is arbitrary. Then

\begin{align}
  ([x,y] - (x,y) t_{\alpha} , t ) & = ([x,y] , t) - (x,y) (t_{\alpha}, t)\nonumber\\
                                  & = (x, [y,t]) - (x,y) \alpha(t)\nonumber\\
                                  & = (x, \alpha(t) y ) - (x,y)\alpha(t)=0.\tag*{\qedhere}
 \end{align}
\end{proof}

We recall that an algebra $\CA$ is called {{\it $G$-graded}}, $G$
an abelian group, if $\CA = \bigoplus_{g \in G} \CA^g$, {where}
each $\CA^g$ {is} a subspace of $\CA$, such that $\CA^g \CA^h
\subseteq \CA^{g +h}$ for all $g, h \in G$.  We will usually
indicate this by saying ``Let $\CA = \bigoplus_{g \in G} \CA^g$ be
a $G$-graded algebra". Each $\CA^g,$ $g\in G,$ is called a {\it
homogeneous  space} and each element of $\CA^g$ a {\it homogeneous
element.} {A subalgebra $\CB$ of $\CA$ is called a {\it graded
subalgebra} if $\CB=\bigoplus_{g \in G} (\CB\cap\CA^g)$.} The {\it
support} of a $G$-graded algebra $\CA$ is the   set supp$_G \CA :=
\left\{ g \in G \mid \CA^g \neq \{0\} \right\}$. We usually use
superscripts to indicate homogeneous spaces, however, when $\CA$
admits two gradings, we use subscripts to distinguish two
gradings, namely $\CA = \bigoplus_{g \in G} \CA^g$ and $\CA =
\bigoplus_{q \in Q} \CA_q$. In this case, we say  $\CA$ admits  a
\emph{compatible $(G,Q)$-grading} if  for all $g\in G,$ $\CA^g =
\bigoplus _{q \in Q} \CA^g_q$ where  $\CA^g_q := \CA^g \cap
\CA_q.$ A bilinear form $\form$ on a $G$-graded algebra $\CA =
\bigoplus_{g \in G}\CA^g$ is called $G$-{\it graded}, if $(\CA^g ,
\CA^h) = \{0\}$ for $g,h\in G$ with $g+h \neq 0$.

\begin{DEF}\label{neher4.5}
  Let $\CA$ be a unital associative algebra.
  An element $a \in \CA$ is called {\it invertible} if there exists a unique element  $a^{-1} \in \CA$
  such that $a a^{-1} = a^{-1} a = 1$.  Suppose $\CA =  \bigoplus_{g \in G} \CA^g$ is $G$-graded{,} then it is called
  \begin{itemize}
    \item \emph{predivision $G$-graded}, if every nonzero homogeneous space contains an invertible element;
    \item \emph{division $G$-graded}, if every nonzero homogeneous element is invertible;
    \item an \emph{associative $G$-torus}{,}
        if $\CA$ is predivision graded and dim $\CA^g \leq
        1$ for all $g \in G$.
  \end{itemize}
\end{DEF}
{We close this section by recalling some  facts from
representation theory of finite groups.}

{Let $G$ be an arbitrary finite group. By $\F[G]$, we mean the
group algebra of $G$ over $\F$. Let $\left\{ \chi_1, \ldots,\chi_n
\right\}$ be the set of {all} irreducible characters of $G$ in
which $\chi_i$ corresponds to {an} irreducible module $V_i$.
Assume $\F$ contains all eigenvalues of all $g \in G$ acting on
$V_i${,} $1\leq i \leq n$. For each $1\leq i \leq n,$ define an
element $e_i$ in $\F[G]$, by
\begin{equation}\label{chi}
  e_i :=  \frac{\chi_i(1)}{|G|} \sum_{g \in G} \chi_i(g^{-1})g,
\end{equation}
{  in which by $|G|$ we mean the order of the group $G$.} It
{follows} that $\{ e_1, \ldots, e_n\}$ forms a complete set of
orthogonal idempotents in $\F[G]$, i.e.
 $e_ie_j=\delta_{ij} e_i$ and $e_1+\cdots +e_n = 1$. So if $M$ is any $\F[G]$-module, then
 \begin{equation}
   M = \bigoplus_{j=1}^n e_j\cdot M.
 \end{equation}
Now if $\pi_j : M\lra e_j\cdot M$ is the projection onto $e_j\cdot
M$, then $\sum_{j=1}^n\pi_j=\id$ and
$\pi_i\pi_j=\delta_{ij}\pi_i$. }

Suppose  now that $G$ is a finite cyclic group of order m, say
$G=\{ 1, \sg, \ldots, \sg^{m-1}\}$. Assume that $\F$ contains an
$m$-th primitive root of unity $\zeta$. Since $G$ is abelian, any
finite dimensional irreducible  $G$-module is one dimensional. Now
it follows that for
$$\chi_j:G\longrightarrow \F,\;\;\sg^i\mapsto \zeta^{ij};\;\;(0\leq i,j\leq m-1),$$  $\chi_0,\ldots,\chi_{m-1}$ form a complete set of
irreducible characters of $G.$ Therefore, if $M$ is any
$\F[G]$-module, we have $M=\bigoplus_{j=0}^{m-1}M_j$, where
$M_j:=\{x\in M\mid \sg(x)=\zeta^jx\}$, and
\begin{equation}\label{projection}
  \pi_j = \frac{1}{m} \sum_{i=0}^{m-1}\zeta^{-ji} \sg^i.
\end{equation}

\section{\bf Gradings induced by automorphisms}\setcounter{equation}{0}\label{gradings}
 In this section, we consider two gradings induced by a finite order automorphism
 on a toral pair, and study their basic properties.
 Let $m$ be a fixed positive integer and suppose $\F$
 contains an {$m$-th} primitive root of unity $\zeta$.
 Throughout this section{,} we assume $(\fg,T)$ is a toral pair, with root system $R,$
 satisfying axioms (IA1) and (IA2) of an IARA.
 Then $\fg = \bigoplus_{\alpha \in R} \fg_\alpha$ where for each $ \alpha \in R$,
 $$\fg_\alpha = \left\{ x \in \fg \mid [t,x] = \alpha(t)x{,}\mbox{ for all }t \in T \right\}. $$
 Also, by (IA1), $\fg$ is equipped with an invariant nondegenerate symmetric
 bilinear form $\form$, such that the form restricted to $T$ is nondegenerate. {It is easy to } see that for any $\alpha, \beta \in R$, $[\fg_\alpha, \fg_\beta ]
\subseteq \fg_{\alpha + \beta}$ and $[\fg_\alpha, \fg_\beta ] =
\{0 \}$ if $\alpha + \beta \notin R$. Also as the form is
invariant,  one sees that
\begin{equation}\label{new-1}
(\fg_\alpha, \fg_\beta) = \{0\}\quad\hbox{unless}\quad \alpha +
\beta = 0,\qquad(\a,\b\in R),
\end{equation}
and concludes that
\begin{equation}\label{newtempeq5}
\form\hbox{ restricted to }\fg_\alpha\oplus \fg_{-\alpha},\;
\alpha\in R,\hbox{ is nondegenerate}.
\end{equation}
 In addition{,} by (IA1) and (IA2)
 for each $\alpha\in R$, there exists {a unique} element $t_\alpha\in T$ such that $\alpha(t)=
 (t,t_\alpha)$ for all $t\in T$.

Now let $\sigma$ be an automorphism of $\fg$ satisfying
\begin{itemize}
    \item[(A1)] $\sigma^m =\hbox{id}_\fg$,
    \item[(A2)] $\sigma(T)=T$,
    \item[(A3)] $(\sigma(x),\sigma(y))=(x,y)$ for all $x,y \in \fg$.
\end{itemize}
For $i\in \Z,$ let $\bar i$ be the image of $i$ in $\Z_m$ under
the canonical map (for the simplicity of notation, we always
denote $\bar{0}$ by $0$). Then setting
\begin{equation}\label{zmgrading}
  \fg^{\bar i}:=\left\{ x\in \fg \mid \sigma(x)=\zeta^i x \right\}
\end{equation}
   for each $i \in \Z$, it is easy to see that  $\fg =\bigoplus_{\bar i\in \Z_m} \fg^{\bar
   i}$ which  defines   a $\Z_m$-grading on $\fg.$ Also by (A2), one can define a similar grading $T=\bigoplus_{\bar i\in\Z_m}T^{\bar
   i}$ on $T$, making
   $T$ into  a graded subalgebra of $\fg.$
    Using $\sigma$, we may define an automorphism, denoted again by $\sigma$, on the vector space $T^\star$
 by $\sigma(\alpha):=\alpha \circ \sigma^{-1}$, {$\a \in T^\star$}. Then
 $\sigma^m=\hbox{id}_{T^\star}$ and so  $\sigma$ induces a $\Z_m$-grading on $T^\star$ as above. One can easily see that for each $\alpha \in R$, $\sigma(\fg_\alpha)=\fg_{\sigma(\alpha)}$. Thus
   \begin{equation}\label{newr}
   \sigma(R)=R.
   \end{equation}
   Note that, {if} $\bar i,\bar j \in \Z_m$, $x\in \fg^{\bar i}$ and $y \in \fg^{\bar j}$, then by (A3),
$(x,y)=(\sigma(x),\sigma(y)) = (\zeta^ix,\zeta^jy) =\zeta^{i+j}(x,y)$.
  Thus $(x,y)=0$ if $\overline{i+j}\neq0$. Consequently
  \begin{equation}\label{newform}
  \hbox{$\form$ is a $\Z_m$-graded bilinear form on $\fg$.}
  \end{equation}

For $\a \in R$, we define $\pa$  to be the restriction of $\a$ to
$T^0$. Since we may consider any element $\b\in (T^{{0}})^\star$
as an element of
  {$T^\star$} by $\b(\sum_{\bar i\not= 0}T^{\bar{i}})=0$, we can consider  $\pa$ as an element of $T^\star$.

For $j\in \Z$, let $\pi_j:\fg\rightarrow \fg^{\bar j}$ be the projection of $\fg$ onto $\fg^{\bar j}$
   with respect to the grading $\fg=\sum_{\bar j\in \Z_m}\fg^{\bar j}$.
   We use the same notation $\pi_j$ for the projection of $T$ onto $T^{\bar j}$, and $T^\star$ onto
   $(T^\star)^{\bar j}$, with respect to the $\Z_m$-gradings
   on $T$ and $T^\star$, respectively. One observes that
   \begin{equation}\label{tempeq3}
   \sigma\circ\pi_j=\pi_j\circ\sigma=\zeta^j\pi_j.
   \end{equation}
{Since the group $\{1,\sg,\ldots,\sg^{m-1}\}$ acts on $\fg$, $T$
and $T^\star$, the following lemma follows immediately from
(\ref{projection}).}
\begin{lem}\label{pi_j}
For any $j \in \Z${,} we have $\pi_j=\frac{1}{m}\sum_{i=0}^{m-1}
\zeta^{-ji}\sigma^i$.
\end{lem}

 For $\alpha\in T^\star$, define
\begin{equation}\label{fgpa}
   \fg_{\pi(\alpha)} := \{ x \in \fg \mid [t,x] = \alpha(t) x,\mbox{ for all}\; t \in T^0 \}.
\end{equation}
Then we have $\fg =\bigoplus_{\pi(\alpha) \in \pi(R)} \fg_{\pi(\alpha)}$ and
\begin{equation}\label{sumgalpha}
   \fg_\pa = \sum_{\{ \beta \in R \mid \pi(\beta) = \pa \}} \fg_\beta; \quad \a \in R.
\end{equation}

\begin{lem}\label{pia=pi0a}
For $\a\in T^\star, $  $\pa=\pi_0(\a)$.
\end{lem}
\begin{proof}
  Suppose $0\leq j\leq m-1$ and $t \in T^{\bar j}$. Then by Lemma~\ref{pi_j}, we have
\begin{eqnarray*}
 \pi_0(\a)(t) &=& \frac{1}{m} \sum_{i=0}^{m-1} \sg^i(\a)(t)\\
                  &=& \frac{1}{m} \sum_{i=0}^{m-1} \a(\sg^{-i}(t))\\
                  &=& \frac{1}{m} (\sum_{i=0}^{m-1} \zeta^{-ji})\a(t).
\end{eqnarray*}
Now since $\zeta$ is a primitive $m$-th root of unity, we have $\sum_{i=0}^{m-1}\zeta^{-ji}=0$ unless $j=0$. Thus $\pi_0(\a)(t)=\a(t)$
 for $t\in T^{0}$ and $\pi_0(\a)(t)=0$ for $t\in \sum_{\bar j\not=0}T^{\bar j}$. Therefore by the way  $\pa$ is defined, we have $\pa=\pi_0(\a)$.
\end{proof}

  We note that $\sigma(\fg_{\pi(\a)})=\fg_{\sigma(\pi(\a))}=\fg_{\pi(\a)},$ $\alpha\in R.$
   Thus  for $\alpha\in R$ and $j \in \Z$,
   \begin{equation}\label{pixinpa}
   \pi_j(\fg_\alpha)\subseteq\pi_j(\fg_{\pi(\alpha)})=\fg^{\bar j}_\pa.
   \end{equation}

   Thanks to  Lemma~\ref{pia=pi0a}, we have
  $   \pi(\a)=\pi_0(\a)=\frac{1}{m}\sum_{i=0}^{m-1}\sigma^i(\a)$ for $\a\in T^\star,$
  so from now on and for the simplicity of notation, we denote all projections $\fg\rightarrow\fg^{0}$, $T\rightarrow T^{0}$ and
   {$T^\star\rightarrow (T^\star)^{0}$}, with respect to the corresponding $\Z_m$-gradings, by $\pi,$ that is
    \begin{equation}\label{newpi}
   \pi=\pi_0=\frac{1}{m}\sum_{i=0}^{m-1}\sigma^i.
   \end{equation}


\begin{lem}\label{pa}
 Let $\gamma\in \hbox{span}_\F R$. {Then $\pi(t_\gamma)=t_{\pi(\gamma)}$ and it is the unique element in $T^0$ satisfying
  $\pi(\gamma)(t)=(t,t_{\pi(\gamma)})$ for all $t\in T^0$.}
\end{lem}

\begin{proof} {First, we note that by (\ref{newr}) and (\ref{newpi}), $\pi(\gamma) \in \hbox{span}_\F R$. Now for} $t\in T$ and $\alpha\in \hbox{span}_\F R$, we have
   \[ (\sigma(t_\alpha),t)=(t_\alpha,\sigma^{-1}(t))=\alpha(\sigma^{-1}(t))=\sigma(\alpha)(t).\]
   Thus $t_{\sigma(\alpha)}=\sigma(t_\alpha)$.
Using this, we are immediately done.
\end{proof}

  {Now ({\ref{sumgalpha}}) together with Lemma \ref{pa} and the same argument as in Lemma  \ref{tpialpha}, gives the following result.}

\begin{pro}\label{toralpair}
The pair  $(\fg, T^0)$ is a toral
pair, with root system $\pi(R)$, satisfying axiom (IA1) of an
IARA. {Moreover, if $\alpha \in R$, $x \in \fg_{\pi(\alpha)}$, $y
\in \fg_{-\pi(\alpha)}$ and $[x , y] \in T^0$, then $[x,y]=(x,y)
t_{\pi(\alpha)}$.}
\end{pro}

{Recall that we now have two gradings on $\fg$, namely the
$\Z_m$-grading induced from automorphism $\sigma$ and the one
induced from the set $\pi(R)$.  For $\alpha\in R$ and $h\in\Z_m$,
set
   \[ \fg^h_{\pi(\alpha)} := \fg^h \cap \fg_\pa. \]
   Since the adjoint action  of $T^0$ stabilizes $\fg^h$ we have
\begin{equation} \label{final2}   \fg^h = \bigoplus_{\pi(\alpha) \in \pi(R)} \fg
^h_{\pi(\alpha)}. \end{equation}
   Thus the following is established.
\begin{lem}\label{QpiR,Hgraded}
The Lie algebra $\fg$ admits a compatible $(\langle\pi(R)\rangle,
\Z_m)$-grading
    \[ \fg = \bigoplus_{\gamma \in \langle\pi(R)\rangle, h \in \Z_m} \fg ^h_\gamma \]
     such that for any $ h \in\Z_m$, $\fg_\gamma^h =\{0\}$ whenever $\gamma \notin \pi(R)$.
\end{lem}
}

\begin{lem}\label{piR-grading} {Let $\alpha,\beta\in R$ and $h,k\in\Z_m$.}

(i) If $\pi(\alpha) + \pi(\beta) \neq 0 $ then $(\fg_{\pi(\alpha)}
, \fg_{\pi(\beta)}) = \{0\}$.

(ii) If $(\fg^h_{\pi(\alpha)}, \fg^k_{\pi(\beta)}) \neq \{0\} $,
then $h+k = 0$ and $\pi(\alpha) + \pi(\beta) = 0$.
\end{lem}
\begin{proof}

(i) Since the form $\form$ is invariant, a standard argument as
{in} the finite dimensional theory, gives the result.

   (ii) It follows from part (i) together with the fact that the form on $\fg$ is {$\Z_m$-graded} and nondegenerate.
\end{proof}

Next, we use Lemma \ref{pa} to define a bilinear form on the
$\F$-span of $\pi(R)$ by
     \[ (\pi(\alpha), \pi(\beta)): = (t_{\pi(\alpha)}, t_{\pi(\beta)}) = (\pi(t_\alpha), \pi(t_\beta)). \]

We conclude this section with the following useful result which
will be used in the sequel. In the following lemma,  in addition
to (IA1) and (IA2), we suppose that  $(\fg, T)$ satisfies (IA3).
\begin{lem}\label{newtemplem6}
Let $(\fg, T)$ be an invariant affine reflection algebra. If $R$
is indecomposable, then
  $\pi(R):=\{\pi(\a)\mid\a\in R\}$ is indecomposable.
\end{lem}

\begin{proof}
We first note that  by Remark \ref{rem}, $R$ is an affine
reflection system. So by \cite[Theorem 1.13]{AYY}, for $\a\in R,$
$\Z\a\subseteq R$ if and only if $\a\in R^0.$ Therefore
$\sigma(\hbox{span}_\F (R^0))\subseteq \hbox{span}_\F (R^0).$ Now one
only needs to adjust the proof of \cite[Proposition 2.6(ii)]{ay}
to our situation.
\end{proof}

\section{\bf Toral pairs and automorphisms }\setcounter{equation}{0}\label{toral-pairs-auto}
In this section{,} we use the same notation as in previous
sections. As in Section \ref{gradings}, we assume that $(\fg,T)$
is a toral pair, with root system $R,$ satisfying axioms (IA1) and
(IA2). We also assume that $\sigma$ is an automorphism of $\fg$
which in addition to axioms (A1)-(A3) satisfies the following
axiom:
\begin{itemize}
  \item[(A4)]$C_{\fg^{0}}(T^{0}):=\{x\in \fg^{0}\mid [t,x]=0;\hbox{ for all } t\in T^{0}\}\subseteq \fg_0$.
\end{itemize}

{Recall that, we have
$$\fg=\sum_{\alpha\in R}\fg_\alpha=\sum_{\alpha\in R}\fg_{\pi(\alpha)}=\sum_{h\in\Z_m}\fg^h=
\sum_{\alpha\in R,\;h\in\Z_m}\fg_{\pi(\alpha)}\cap\fg^h,$$ and
$T=\sum_{h\in \Z_m}T^h.$}

For $\alpha \in R,$ let $\ell_\sigma(\alpha)$ be the least
positive integer such that
$\sigma^{\ell_\sigma(\alpha)}(\alpha)=\alpha$, then
$\ell_\sigma(\alpha)\mid m$ and we have the following lemma which
gives an equivalent condition to (A4). The proof of this lemma is
essentially similar to the proof of \cite[Proposition
3.25]{abpcovering1}{,} however  for the convenience of the reader,
we provide a proof here.

\begin{lem}\label{newtemplem9}
  (A4) is equivalent to (A4)$'$ below:
  \begin{itemize}
    \item[(A4)$'$] For $\alpha \in R\backslash\{0\},$ either $\pi(\alpha) \neq 0 $ or
    $\left\{ x \in \fg_\alpha \mid \sigma^{\ell_\sigma(\alpha)}(x) = x \right\} = \{0\}$.
  \end{itemize}
  Moreover if $m$ is prime, then (A4) and (A4)$'$ are equivalent to
  \begin{itemize}
    \item[(A4)$''$] $\pi(\alpha) \neq0 $ {for every $\alpha \in R\setminus\{0\}$.}
  \end{itemize}
\end{lem}
\begin{proof}
  Suppose (A4) holds but (A4)$'$ fails, then there exist $\alpha\in R \backslash \{0\}$
  and $0\neq x\in \fg_\alpha$ such that $\pa=0$ and $\sigma^{\ell_\sigma(\alpha)}(x)=x$.
  Abbreviate $\ell_\sigma(\alpha)$ by $\ell$ and let $y :=x +\sigma(x)+\cdots+\sigma^{\ell-1}(x)$,
  then $\sigma(y)=y$ and so
   $y\in\fg^{0}$. Also since the elements $\sigma^i(x)$  ($0\leq i\leq \ell-1$) belong to
   different root spaces, $y\neq 0$. In addition $y\in \fg_\pa=\fg_{\pi(0)}=C_\fg(T^0),$
   so $y\in C_{\fg^{0}}(T^0)\subseteq\fg_0$ which is a contradiction as $y\notin \fg_0$.

  Conversely, assume (A4)$'$ holds and let $x=\sum_{\alpha\in R}x_\alpha\in C_{\fg^{0}}(T^{{0}})$,
  where $x_\alpha\in \fg_\alpha$. Since $\sigma(x)=x$, $\sigma(x_\alpha)=x_{\sigma(\alpha)}$
  for any $\alpha\in R$, therefore $\sigma^{\ell_\sigma(\alpha)}(x_\alpha)=x_\alpha$. Thus by (A4)$'${,}
  for any $0\neq\alpha\in R$ with $\pa=0$, $x_\alpha=0$. On the other hand, for every $t\in T^0$,
  $0=[t,x]=\sum_{\alpha\in R}\pi(\alpha)(t)x_\alpha$. Hence $x_\alpha=0$ for any $\alpha\in R \backslash \{0\}$
  with $\pa\neq0$ and so $x=x_0\in\fg_0$.

  Finally, suppose that $m$ is a prime number. Clearly it suffices to show that (A4)$'$ implies (A4)$''$.
  Suppose to the contrary that $\pa=0$ for some nonzero $\alpha \in R$. By Lemma~\ref{pi_j},
  $\sigma(\alpha) \neq \alpha,$ so $\ell_\sigma(\alpha) \neq 1.$ Now as  $\ell_\sigma(\alpha)$
  divides $m$ and $m$ is prime, we have $\ell_\sigma(\alpha) =m$. Hence
  $\sigma^{\ell_\sigma(\alpha)}(x_\alpha)= x_\alpha$ for all $x_\alpha \in \fg_\alpha$ which contradicts (A4)$'$.
\end{proof}

\begin{lem}\label{newtemplem1}
Suppose $\alpha,\beta\in R$ with $\alpha\not=\beta$ and
$\pi(\alpha)=\pi(\beta)$. If $x\in\fg_\alpha$ and
$y\in\fg_{-\beta}$, then $\pi([x,y])=0$.
\end{lem}

\begin{proof}
If $\a-\b \not\in R$, there is nothing to prove, so suppose $\a-\b \in R$. We have
$$[x,y]\in\fg_{\alpha-\beta}\subseteq\fg_{\pi(\alpha-\beta)}=\fg_{\pi(0)}=C_{\fg}(T^{
0}).$$ Therefore, $\sigma^i([x,y])\in C_{\fg}(T^{ 0})$, for all
$i$, and so $\pi([x,y])\in C_{\fg^{0}}(T^{0})$. Thus by (A4),
$\pi([x,y])\in \fg_0$. On the other hand, $\sigma^i([x,y])\in
\fg_{\sigma^i(\alpha-\beta)},$  for all $i,$ also as
$\alpha-\beta\neq0,$ we have $\sigma^i(\alpha-\beta)\not=0$. So
$\pi([x,y])$ is a sum of elements, each belong{s} to a root space
corresponding to a nonzero root. But since $\pi([x,y])\in\fg_0$,
this can  happen only  if $\pi([x,y])=0$.
\end{proof}

\begin{lem}\label{[pi_j,pi_-j]}
{(i) For $x,y\in \fg$ and $j,k\in\Z$, we have
\[
  [\pi_j(x),\pi_{k}(y)]= \pi_{j+k}([x,\pi_k(y)]).
\]
In particular,
$$ [\pi_j(x),\pi_{-j}(y)]= \pi([x,\pi_{-j}(y)])
  = \frac{1}{m}\sum_{i=0}^{m-1}\pi([x,\zeta^{ji}\sg^{i}(y)]).
$$
}

{(ii) If {$\alpha \in R$}, $x \in \fg_\alpha$, $y \in
\fg_{-\alpha}$ and $\ell:=\ell_\sigma(\alpha)$, then for $j\in\Z$,
$$ [\pi_j(x),\pi_{-j}(y)] =\frac{1}{m}\sum_{i=0}^{(m/\ell)-1}\pi([x,\zeta^{ji\ell}\sg^{i\ell}(y)]).
$$}

(iii)  Let $\alpha, \beta $ belong to distinct $\sigma$-orbits of
$R$ {with}
  $\pa=\pi(\beta)$. If $x \in \fg_\alpha, y \in \fg_{-\beta}$,
  then $[\pi_j(x),\pi_{-j}(y)]=0$.
\end{lem}
\begin{proof}
(i) It  is clear, since $\pi_j$ is the projection onto $\fg^{\bar
j}$ with respect to  $\Z_m$-gradation of $\fg$.

 (ii) Assume that $\alpha$, $x$, $y$ and $\ell$ are as in the statement. By part (i),
 $$[\pi_j(x),\pi_{-j}(y)]=\frac{1}{m}\big(\sum_{\{0\leq t\leq m-1:\; \ell\mid t\}}\pi([x,\zeta^{jt}\sigma^t(y)])+\sum_{\{0\leq t\leq m-1:\; \ell\nmid t\}}\pi([x,\zeta^{jt}\sigma^t(y)])\big).$$
 So it is enough to show that $\pi([x,\sigma^t(y)])=0$ for all
 $0\leq t\leq m-1$ with $\ell\nmid t$. Assume that $\ell\nmid t$. Then $\alpha-\sigma^t(\alpha)\neq 0$, $x\in \fg_\alpha$, $\sg^t(y) \in \fg_{-\sigma^t(\alpha)}$ and $ \pa = \pi(\sigma^t(\alpha))$, thus by Lemma \ref{newtemplem1},
 $\pi([x,\sigma^t(y)])=0$.

 (iii)  By part (i), $[\pi_j(x),\pi_{-j}(y)]=
  \left(1/m\right)
  \sum_{i=0}^{m-1}\pi([x,\zeta^{ji}\sigma^i(y)])$.
By the  assumption, for any $i$, $\alpha-\sigma^i(\beta)\neq 0$
and so by Lemma \ref{newtemplem1},
$\pi([x,\zeta^{ij}\sigma^i(y)]=0$, hence
$[\pi_j(x),\pi_{-j}(y)]=0$.
 \end{proof}

{Next, consider $\hbox{Aut}(\fg)$, the automorphism group of
$\fg$.
 One knows that  the subgroup $(\sg)$ of $\Aut(\fg)$ generated by $\sigma$, acts naturally on $R$.
We call any orbit of this action, a {\it $\sg$-orbit}.  Then two
roots $\alpha,\beta$ belong
  to the same $\sg$-orbit if and only if $\sigma^i(\alpha)= \beta$, for some
  $i$. Fix a set $\hbox{orb}(R)$ of distinct representatives for all $\sg$-orbits,
  namely $R=\uplus_{\alpha\in\hbox{orb}(R)}(\sg)\cdot\alpha$.}
 The following two lemmas are of great importance for our goal.

\begin{lem}\label{templem2}
 Let $0\leq j\leq m-1$.

  (i) If $\alpha,\beta\in R$ belong to  the same $\sigma$-orbit, then
  $\pi_j(\fg_\alpha)=\pi_j(\fg_\beta)$.

  (ii) For $\alpha\in R$,
  $$\fg^{\bar j}_{\pi(\alpha)}
  =\sum_{\{\beta\in\hbox{orb}(R)\mid
  \pi(\beta)=\pi(\alpha)\}}\pi_j(\fg_{\beta}).$$
\end{lem}

\begin{proof}
  (i) Suppose $\beta=\sigma^n(\alpha)$, $n\in\Z$. By
  (\ref{tempeq3}), $\pi_j\circ\sigma^n=\zeta^{nj}\pi_j$. Therefore
$$
  \pi_j(\fg_\beta)=\pi_j(\fg_{\sigma^n(\alpha)})=\pi_j\sigma^n(\fg_\alpha)=
  \zeta^{nj}\pi_j(\fg_\alpha)=\pi_j(\fg_\alpha).
$$

{  (ii) By Lemma \ref{QpiR,Hgraded}, for every $1\leq j \leq m-1$ and every $ \alpha \in R$ we have $\fg_\pa^{\bar j} = \pi_j(\fg_\pa)$.
  Now this together with (\ref{sumgalpha}) implies that
\begin{equation}\label{gpabarj}
   \fg_\pa^{\bar j}=\sum_{\{\beta\in R\mid\pi(\beta)=\pa\}} \pi_j(\fg_{\beta}).
\end{equation}
}
  and so the result follows immediately from part (i).
\end{proof}

{ Let {$\a\in R$}, $\ell:=\ell_\sg(\a)$ and $j\in\Z$. For
$x\in\fg_\a$, we set
\begin{equation}\label{xj}
\bar{x}_j:=\sum_{i=0}^{(m/\ell)-1}\zeta^{-ji\ell}\sg^{i\ell}(x)\in\fg_\a.
\end{equation}
Note that the implication $\bar{x}_j\in\fg_\a$ follows from the fact that $\sg^\ell(\fg_\a)=\fg_{\sg^\ell(\a)}=\fg_\a$.
The following observation is a key result for the rest of the work.
}

\begin{lem}\label{pi_-j(y)neq0}
  {Suppose {$\alpha\in R$}, $\ell:=\ell_\sigma(\alpha)$,  $x\in \fg_{\alpha}$ and $j\in \Z$. Then}

{(i)  $\pi_j(x)=(1/m)\sum_{i=0}^{\ell-1}\zeta^{-ji}\sigma^i(\bar
x_j)$,}

{(ii) $\pi_{j}(x)\neq 0$ if and only if $\bar{x}_j\not=0$.}
\end{lem}

\begin{proof}
{Set $k:=(m/\ell)-1$. Using Lemma~\ref{pi_j}, we have } {
\begin{eqnarray*}
  m\pi_{j}(x)&=&\sum_{i=0}^{m-1}\zeta^{-ji}\sigma^i(x)\\
             &=&\sum_{i=0}^{\ell-1}\zeta^{-ji}\sigma^i(x)+\sum_{i=\ell}^{2\ell-1}\zeta^{-ji}\sigma^i(x)+\cdots+ \sum_{i=k\ell}^{m-1}\zeta^{-ji}\sigma^i(x)\\
             &=&\sum_{s=0}^{k}\sum_{i=0}^{\ell-1}\zeta^{-j(s\ell+i)}\sigma^{s\ell+i}(x)\\
             &=&\sum_{i=0}^{\ell-1}\zeta^{-ji}\sigma^i\left(\sum_{s=0}^k\zeta^{-js\ell}\sigma^{s\ell}(x)\right)\\
             &=&\sum_{i=0}^{\ell-1}\zeta^{-ji}\sigma^i(\bar x_j).
\end{eqnarray*}
This proves (i).}

{(ii) Since for each $0\leq i\leq \ell-1$, $\sigma^i(\bar
x_j)\in\fg_{\sigma^i(\alpha)}$ and
$\alpha,\sigma(\alpha),\ldots,\sigma^{\ell-1}(\alpha)$ are
distinct roots, we concluded that
$\sum_{i=0}^{\ell-1}\zeta^{-ji}\sigma^i(\bar x_j) =0$ if and only
if $\bar x_j =0$. Therefore using part (i), we are done.}
\end{proof}

\section{\bf Division IARA's and automorphisms}\setcounter{equation}{0}\label{division IARA}
In this section, we use the same notation as in previous sections.
We also  assume that $(\fg,T)$ is a division IARA with root system
$R$, that is, $(\fg,T)$ satisfies axioms (IA1), (IA2)$'$ and
(IA3). Further suppose that $\sigma$ is an automorphism of $\fg$
satisfying (A1)-(A4). In Section \ref{toral-pairs-auto}, we saw
hat $(\fg, T^0)$ is a toral pair satisfying axiom (IA1), and
established several other properties of $(\fg,T^0)$.
 Our main aim in this section is to show
that $(\fg,T^0)$ is an IARA with root system $\pi(R)$. This in
particular implies that  $\pi(R)$ is an affine reflection system.

\begin{lem}\label{[],()}
Let {$\a \in R$}, $x \in \fg_\a$ and $y \in \fg_{-\a}$. If $j\in
\Z$ and $\bar x_j$ is defined as in
 (\ref{xj}), then

  {(i)  $[\pi_j(x),\pi_{-j}(y)]=(1/m)\pi([\bar x_j,y])$,}

  {(ii) $(\pi_j(x),\pi_{-j}(y))=(1/m)(\bar x_j,y)$.}
\end{lem}

\begin{proof}
(i) {Let $k:=(m/\ell)-1$. By Lemma \ref{[pi_j,pi_-j]}, replacing
$j$ with $-j$, $\alpha$ with $-\alpha$ and $x$ with $y$, we have
\begin{eqnarray*}
  [\pi_j(x),\pi_{-j}(y)]&=& \frac{1}{m} \sum_{i=0}^{k} \pi([\zeta^{-j\ell i}\sigma^{\ell i}(x),y])\\
                        &=& \frac{1}{m} \pi([\sum_{i=0}^{k}\zeta^{-j\ell i}\sigma^{\ell i}(x),y])\\
                        &=& \frac{1}{m}  \pi([\bar x_j, y]).
\end{eqnarray*}
(ii) By Lemma \ref{pi_-j(y)neq0}(i), $\pi_j(x) =
(1/m)\sum_{i=0}^{\ell-1}\zeta^{-ji} \sg^i(\bar x_j)$ and
$\pi_{-j}(y) = (1/m) \sum_{i=0}^{\ell-1}\zeta^{ji} \sg^i(\bar
y_{-j})$. Also, using the definition of $\ell:=\ell_\sg(\a)$ and
(\ref{new-1}), we see that
$(\fg_{\sg^i(\a)},\fg_{\sg^{j}(-\a)})=\{0\}$, if $0\leq
i\not=j\leq \ell-1$. Hence {\begin{align*}
  (\pi_j(x),\pi_{-j}(y)) &= \frac{1}{m^2} \sum_{i=0}^{\ell-1}(\sg^{i}(\bar x_j),\sg^{i}(\bar y_{-j}))\\
                         &= \frac{1}{m^2} \ell (\bar x_j,\bar y_{-j})\\
                         &= \frac{1}{m^2} \ell \sum_{i=0}^{k} (\bar x_j, \zeta^{ji\ell}\sg^{i\ell}(y))\\
                         &= \frac{1}{m^2} \ell \sum_{i=0}^{k} \zeta^{ji\ell} (\sg^{-i\ell}(\bar x_j),y)\\
                         &= \frac{1}{m^2} \ell \sum_{i=0}^{k} (\bar x_j,y)\\
                         &= \frac{1}{m}(\bar x_j,y).\tag*{\qedhere}
\end{align*}
}}
\end{proof}
\begin{lem}\label{beforstar}
   Let $\alpha\in R$ with $\pa\neq 0$. Suppose $x \in \fg_{\alpha}$ {and
 $\pi_{j}(x)\neq 0$, for some $j\in\Z$. Then} there exists $y\in \fg_{-\alpha}$ such that $0\neq [\pi_j(x),\pi_{-j}(y)]\in T^0$.
\end{lem}
\begin{proof}
{Contemplating (\ref{xj}), Lemma \ref{pi_-j(y)neq0} implies that
$\bar x_j$ is a nonzero element of $\fg_\a$. Since by our
assumption, the axiom (IA2)$'$
  holds for $(\fg, T)$, there exists $y\in \fg_{-\alpha}$ such that $0\neq [\bar x_j, y]\in T$. Therefore, by Lemma \ref{tpialpha},
\begin{equation}\label{tempeq1}
(\bar x_j,y)\not=0.
\end{equation}
Now combining this, Lemmas \ref{[],()}, \ref{tpialpha} and
\ref{pa}, we get
$$[\pi_j(x),\pi_{-j}(y)]=\frac{1}{m}(\bar x_j, y) t_\pa\in T^0.$$
But as $\pi(\alpha)\not=0$, we have $t_\pa\not=0$, and so we are
done by (\ref{tempeq1}).}
\end{proof}

\begin{lem}\label{star}
Let $\alpha\in R$ with $\pi(\alpha)\not=0$, {and $j\in\Z$}.  Then
for every $0\neq e_{\pi(\alpha)}^{\bar j} \in
\fg_{\pi(\alpha)}^{\bar j}$ there exists $f_{\pi(\alpha)}^{\bar j}
\in \fg_{-\pi(\alpha)}^{-\bar j}$ such that
  $0 \neq [e_{\pi(\alpha)}^{\bar j} ,f_{\pi(\alpha)}^{\bar j}] \in T^0$. In particular, axiom (IA2) holds for
  the toral pair $(\fg, T^0)$.
\end{lem}
\begin{proof}
  By Lemma \ref{templem2}, $e_\pa^{\bar j} = \pi_j(x_1) + \cdots + \pi_j(x_n)$
where $x_i\in \fg_{\alpha_i}$ for some $\alpha_i$'s belong to
distinct $\sigma$-orbits of $R$, satisfying $\pi(\alpha_i)=\pa$,
for all $i$. Thus for some $1\leq i \leq n$, $\pi_j(x_i)\neq 0$,
and by Lemma \ref{beforstar}, there exists
  $y \in \fg_{-\alpha_i}$ such that $0\neq [\pi_j(x),\pi_{-j}(y)] \in T^0$.  So using
\rred{Lemma~\ref{[pi_j,pi_-j]}~(iii)}, we have
\begin{eqnarray*}
  [e_\pa^{\bar j},\pi_{-j}(y)]= [\pi_j(x_i),\pi_{-j}(y)]\in T^0\setminus
  \{0\}.
\end{eqnarray*}
Now setting  $f_\pa^{\bar j}:=\pi_{-j}(y),$ we get the first
assertion as by (\ref{pixinpa}), $\pi_{-j}(y) \in
\fg^{-j}_{-\pa}.$ To see the final assertion in the statement, let
$\alpha\in R$ with $\pi(\alpha)\not=0$. As
$0\not=\fg_\alpha\subseteq\fg_{\pi(\alpha)}= \sum_{\bar
j\in\Z_m}\fg_{\pi(\alpha)}^{\bar j}$, we have
$\fg_{\pi(\alpha)}^{\bar j}\not=0$ for some $\bar j$. {Now by the
first part of the statement}, there exist $e^{\bar
j}_{\pi(\alpha)}\in\fg^{\bar j}_{\pi(\alpha)}$ and $f^{\bar
j}_{\pi(\alpha)}\in\fg^{-\bar j}_{-\pi(\alpha)}$ such that
$0\not=[e^{\bar j}_{\pi(\alpha)},f^{\bar j}_{\pi(\alpha)}]\in
T^0$. This means that  (IA2) holds for $(\fg, T^0)$.
\end{proof}

We are now ready to state the main result of this section, {which
extends \cite[Theorem 3.4 ]{ay} to a rather larger class.}
\begin{thm}\label{gt0isiara}
  Let $(\fg, T)$ be a division IARA with corresponding root system $R$ and bilinear form $\form$. Suppose $\sigma$
  is an automorphism of $\fg$ satisfying (A1)-(A4){, and $T^0$ is} the set of fixed points of
  $\sigma$ on $T$. For $\alpha\in R$, {let $\pa$ be the restriction of $\a$ to $T^0$.} Then $(\fg, T^0)$
  is an IARA with root system $\pi(R): = \left\{ \pa \mid \alpha \in R \right\}$.
  In particular, $\pi(R)$ is an affine reflection system. Moreover, if $R$ is indecomposable, then so is $\pi(R)$.
\end{thm}

\begin{proof}
 We have shown in Lemma \ref{toralpair} that $(\fg, T^0)$ is a toral pair such that
  $\fg = \bigoplus_{\pa \in \pi(R)} \fg_\pa$, and that axiom (IA1) of Definition \ref{iara} holds for $(\fg, T^0)$.
   Also by Lemma \ref{star}, (IA2) holds for  $\fg$. So it remains to prove (IA3).

  Let $\alpha,\beta\in R$ with $(\pi(\alpha),\pi(\alpha))\not=0$, $x\in\fg_{\pi(\alpha)}$ and $y\in\fg_{\pi(\beta)}$. We must show
  that  $\hbox{ad}(x)^n(y)=0$ for some $n$. We know that $\hbox{ad}(x)^n(y)\in \fg_{n\pa+\pi(\beta)},$ so
  \[ [t_\pa,(\hbox{ad}x)^n(y)]=(n\pa+\pi(\beta))(t_\pa) (\hbox{ad}x)^n(y).\]
  {Therefore  if $(\hbox{ad}x)^n(y)$ is nonzero, it is an eigenvector for $\hbox{ad} t_\pa$ with eigenvalue
  $(n\pi(\alpha)+\pi(\beta))(t_{\pi(\a)}).$ But  for distinct values of $n,$
 the scalers $(n\pi(\alpha)+\pi(\beta))(t_{\pi(\a)})$ are distinct, so it is
 enough to show that
  $\hbox{ad}t_\pa$ has a finite number of eigenvalues as an operator on $\fg$. One knows that
  each
  eigenvalue of $\ad t_\pa$ on $\fg$ is of the form $\pi(\gamma)(t_\pa)$ for some $\gamma\in R$, and by Lemma~\ref{pa},
  \[\pi(\gamma)(t_\pa)=(\pi(\gamma),\pa)=(\gamma,\pa)\subseteq\frac{1}{m}(\underbrace{A+A+\cdots+A}_{\hbox{m-times}})\]
  where $A:=\{ (\gamma,\beta)\mid \gamma,\beta\in R\}.$}
 Now since $R$ is an affine reflection system, the set
 $A$ is finite; see \cite[Sections 3.7,3.8]{nehersurvey} {and \cite[Theorem 8.4]{nllfrs}}.
Therefore $\ad t_\pa$ has only a finite number of eigenvalues.
These all together show  that $(\fg, T^0)$ is an IARA. Thus its
root system $\pi(R)$ is an affine reflection system, by
\cite[Theorem 6.8]{nehersurvey}. {T}he final assertion of the
statement follows from Lemma \ref{newtemplem6}.
\end{proof}

\begin{rem}\label{sl2}
  {Suppose  $\pi(\alpha) \in \pi(R)^\times$ and $h\in\Z_m$. By Lemmas \ref{star} and \ref{tpialpha}, we may {choose}
  $e_{\pi(\alpha)}^h\in\fg_\pa^h$ and $f_{\pi(\alpha)}^h\in\fg^h_{-\pa}$
 such that $[e_{\pa}^h, f^h_{\pa}]=(e^h_\pa, f^h_{\pa})t_{\pa}\not=0$.}
  {So  multiplying $f_{\pi(\alpha)}^h $ by $2/((e_{\pi(\alpha)}^h, f_{\pi(\alpha)}^h )({\pi(\alpha)},{\pi(\alpha)}))$
  we have
  \[[e_{\pi(\alpha)}^h ,f_{\pi(\alpha)}^h] = \frac{2 t_{\pi(\alpha)}}{(\pi(\alpha), \pi(\alpha))}.\]
Thus  setting  $h_\pa:=2 t_{\pi(\alpha)} / (\pi(\alpha),
\pi(\alpha))$, the triple  $\{ e_{\pi(\alpha)}^h
,h_\pa,f_{\pi(\alpha)}^h\}$ forms an $\mathfrak{sl}_2$-triple.}
\end{rem}

\begin{lem}\label{(pij,pi-j)neq0}
{Let $j\in\Z$, $\a\in R\setminus\{0\}$, $\pi(\a)=0$ and
$\pi_j(\fg_\a)\not=\{0\}$.}

{(i) For each $x\in\fg_\a$ with $\pi_j(x)\not=0$, there exists
$y\in\fg_{-\a}$ such that $[\pi_j(x),\pi_{-j}(y)]=0$, but
$(\pi_j(x),\pi_{-j}(y))\not=0$.}

{(ii) There exists $e \in \fg_{\pi(0)}^{\bar j}$ and $f \in
\fg_{\pi(0)}^{-\bar j}$ such that
  $[e,f] = 0$ but
  $(e,f) \neq 0$.}
\end{lem}

\begin{proof}
{(i) Let $x\in\fg_\a$ and $\pi_j(x)\not=0$. By Lemma
\ref{pi_-j(y)neq0}(ii), we have $0\not=\bar x_j\in\fg_\a$. Since
(IA2)$'$ holds for $(\fg,T)$, there exists $y \in \fg_{-\a}$ such
that $0 \neq [\bar x_j,y]\in T$. Therefore, by Lemma
\ref{tpialpha}, $(\bar x_j,y)\not=0$.  Now this, together with
Lemma \ref{[],()}(ii), gives
$$(\pi_j (x),\pi_{-j}(y))=(1/m)(\bar x_j,y)\not=0.$$}
{On the other hand, combining Lemmas \ref{[],()}(i),
\ref{tpialpha} and \ref{pa}, we obtain
$$m[\pi_j(x),\pi_{-j}(y)]=\pi_j([\bar x_j,y])=(\bar x_j,y)\pi(t_{\a})=(\bar x_j,y)t_{\pi(\a)}=(\bar
x_j,y)t_0=0.$$}

{(ii) By assumption, $\pi_j(\fg_\a)\not=0$. So $\pi_j(x)\not=0$
for some $x\in\fg_\a$. Now taking $e:=\pi_j(x)\in\fg^{\bar
j}_{\pi(0)}$ and $f:=\pi_{-j}(y)\in\fg^{-\bar j}_{\pi(0)}$ as in
part (i), we are done.}
\end{proof}

As it will be revealed from the sequel, if $\fg_0$ is abelian, the
axioms (A1)-(A4) imposed on the automorphism $\sg$, are enough for
our purposes in this work. However, this is not the case for a
general IARA. To be more precise, we note that the main difference
of the class of invariant affine reflection algebras with extended
affine Lie algebras {or} locally extended
affine Lie algebras, 
is that in the latter ones, the subspaces $T$ and $\fg_0$
coincide, while in an IARA, $T$ might be a proper subspace of
$\fg_0$. This in particular, forces $\fg_0$ not to be necessarily
abelian. In this case, to have a control on the action of $\sg$ on
the pair $(\fg_0,T)$, we need the following ``tameness condition''
whose offshoot is given in Lemma \ref{afterstar}. {\begin{itemize}
\item[(A5)] If $\{0\}\not=\fg_{\pi(0)}^{\bar
    j}\subseteq\fg_0$, then $T^{\bar j}\not=\{0\}$,
    $j\in\Z$.
\end{itemize}}

\begin{lem}\label{afterstar}
  {Suppose $\sg$ satisfies (A1)-(A4). Also suppose that  $\fg_0$ is abelian or  (A5) holds for  $\sg.$ If $j \in \Z$ and $\fg_{\pi(0)}^{\bar j} \neq\{ 0\}$,
  then there exist $e\in\fg_{\pi(0)}^{\bar j}$ and $f\in\fg_{\pi(0)}^{-\bar j}$
  such that $[e,f]=0$, but $(e,f)\not=0$.}
\end{lem}
\begin{proof}
{Assume $j\in\Z$ and $\fg^{\bar j}_{\pi(0)}\not=\{0\}$. By
(\ref{sumgalpha}), $\fg^{\bar j}_{\pi(0)}=\sum_{\{\a\in R\mid
\pi(\a)=0\}}\pi_j(\fg_\a)$. If $\pi_j(\fg_\a) \neq 0$ for some
nonzero root $\a$ with $\pa=0$, we are done by
Lemma~\ref{(pij,pi-j)neq0}. Otherwise, $\{0\}\not=\fg^{\bar
j}_{\pi(0)}=\fg^j_0=\pi_j(\fg_0)\subseteq\fg _0$. Now if $\fg_0$
is abelian, then since $\form$} {is nondegenerate and
$\Z_m$-graded on $\fg_0$, there exists $e \in \fg_0^{\bar j}$ and
$f \in \fg_0^{-\bar j}$ such that
  $(e,f)\neq 0$ but as $\fg_0$ is abelian $[e,f]=0$. If (A5) holds, then, $T^{\bar
j}=\pi_j(T) \neq 0$. Since $\form$ is nondegenerate and
  $\Z_m$-graded on $T$, there exist $e \in T^{\bar j}$ and $f \in T^{-\bar j}$ such that
  $(e,f)\neq 0$ but as $T$ is abelian $[e,f]=0$.}
\end{proof}

{Assumption (A5) (Lemma \ref{afterstar}) will be used to prove
condition (IA2) holds for a Lie algebra $\hat\fg$ which will be
introduced in Section~\ref{extended affinization}.}

\section{\bf Fixed point subalgebras of IARA's}\setcounter{equation}{0}
An interesting subject of research on algebras is the study of
subalgebra of points which are fixed by certain types of
automorphisms. The starting point of such a study, in our context,
is the work of Borel and Mostow \cite{bm} on semisimple Lie
algebras. They showed that the subalgebra of fixed points of a
finite order automorphism of a semisimple Lie algebra is a
reductive Lie algebra. Motivated by this work, in \cite{aby}, the
authors  showed that the fixed point subalgebra of an extended
affine Lie algebra is a sum of extended affine Lie algebras (up to
existence of some isolated root spaces), a subspace of the center
and a subspace which is contained in the centralizer of the core.
They also showed that the core of the fixed point subalgebra
modulo its center is isomorphic to the direct sum of the cores
modulo centers of the involved summands. In \cite{ak}, the authors
did a similar study on the fixed points  of a Lie torus under
certain automorphism and obtained some similar results. In
\cite{youfixedpoint}, the author considered the same study {in} a
rather more general context, namely root graded Lie algebras. She
proved that the core of the subalgebra of fixed points of a root
graded Lie algebra under a suitable automorphism is the sum of a
root graded Lie algebra $\mathcal L$ and a subspace $\mathcal K$
whose normalizer contains $\mathcal L$.

 {We now consider the same question for an IARA, namely what is the structure of fixed points
  of a division IARA $(\fg, T)$ under an automorphism $\sg$ satisfying axioms (A1)-(A4).
  We will show, using the results of the previous
  sections, that this subalgebra is a division IARA with toral
  subalgebra $T^0$. Since conditions (A1)-(A4) introduced in \cite{aby} and \cite{ak}
  coincide with conditions (A1)-(A4) given here, the following theorem generalizes and at the same time gives a new
  perspective to some of the results there.}

\begin{thm}\label{g0t0isiara}
  Let $(\fg, T)$ be a division IARA with corresponding root system $R$ and bilinear form $\form$. Suppose $\sigma$
  is an automorphism of $\fg$ satisfying (A1)-(A4) and $\fg^0$ (resp. $T^0$) is the set of fixed points of
  $\sigma$ on $\fg$ (resp. $T$). Then $(\fg^0, T^0)$  is a division IARA with root system
  \begin{equation}\label{rsg}
  R^\sg:=\{\pi(\a)\mid\a\in R,\;\fg^0_{\pi(\a)}\not=0\}.
  \end{equation}
  In particular, $R^\sg$ is an affine reflection system.
\end{thm}
\begin{proof}
  By Lemma~\ref{QpiR,Hgraded},
  $$\fg^0 = \bigoplus_{\pa\in \pi(R)} \fg^0_{\pa}=\bigoplus_{\tilde\a\in R^\sg}\fg^0_{\tilde\a}$$
  where $R^\sg$ is given by (\ref{rsg}). So $(\fg^0,T^0)$ is a toral pair.
  In addition, since by (\ref{newform}) the form $\form$
  is $\Z_m$-graded on $\fg$, it is nondegenerate  on both $\fg^0$ and $T^0$, therefore (IA1) holds.
 Also  (IA2)$'$ holds by Lemma~\ref{star}. Next let $\a\in R^\sigma$ with $(\pi(\a),\pi(\a))\not=0$, and
  $x\in \fg^0_{\pi(\a)}$. By Theorem \ref{gt0isiara}, $(\fg, T^0)$ is an IARA and so (IA3) holds for $(\fg, T^0)$.
Therefore as $\fg^0_{\pi(\a)}\sub\fg_{\pi(\a)}$, $\ad x$ is
locally nilpotent on $\fg$ and so on $\fg^0$. This shows that
(IA3) holds for $(\fg^0,T^0)$ and so $(\fg^0, T^0)$ is a division
IARA. Now $R^\sg$ as the root system of an IARA is an affine
reflection system.
\end{proof}

\begin{rem}\label{newrem-5}
By Theorems \ref{gt0isiara} and \ref{g0t0isiara}, both $\pi(R)$
and $R^\sg$ are affine reflection systems with $R^\sg\sub\pi(R)$.
It is shown in \cite{aby} that $R^\sg$ might be a proper subset of
$\pi(R)$, and in fact in many examples this is the case. It is
worth mentioning that $R^\sg$ and $\pi(R)$ might not be
necessarily of the same type, see \cite[Example 3.70]{aby}.
\end{rem}

\section{\bf Extended Affinization}\setcounter{equation}{0}\label{extended affinization}
 In this section, we study {\it extended affinization,} a process
in which starting from an IARA $\fg$ with root system $R$ and  a
finite order automorphism of $\fg$, we get a new IARA whose root
system is  an extension of $\pi(R)$ (see (\ref{newpi})). The
notion of {\it affinization} was initiated by V. Kac~\cite{kac} in
order to realize affine Kac-Moody Lie algebras. Since then, this
method has been used by different authors to realize certain
generalizations of affine Lie algebras, e.g. in
\cite{abpcovering1}, the authors use this method to realize
extended affine Lie algebras, also  in \cite{abfpgs} and
\cite{abfpeala}, using this method, the authors realize Lie tori.

Throughout this section,  $(\fg, T)$ is an IARA with root system
$R$, $\sigma$ an automorphism of
  $\fg$ satisfying (A1)-(A3), and $T^0$ the set of fixed points of $\sigma$ on $T$.
We recall that for  $\alpha\in R$,  $\pa$ is  the restriction of
$\alpha$ to ${T^0}$ and that we have a
$(\langle\pi(R)\rangle,\Z_m)$-grading on $\fg$ as in
Lemma~\ref{QpiR,Hgraded}.
  Suppose $\Lambda$ is a torsion free abelian group and let $\rho:\Lambda\lra \Z_m$ be a group epimorphism. For $\lambda\in\Lambda,$ we take $\bar{\lambda}:=\rho(\lambda).$

  Suppose $\CA$ is a unital commutative associative algebra. In addition,
   suppose $\CA = \bigoplus_{\lambda \in \Lambda} \CA^\lambda$ is predivision $\Lambda$-graded. It is easy to see that in this case
   $\supp_\Lam(\CA)$ is a subgroup of $\Lam.$ Since the $\Lam$-grading of $\fg$ depends only on  $\supp_\Lam(\CA),$ we may assume without loss of
   generality that
$\Lam=\supp_\Lam(\CA)$, that is,
  \begin{equation}\label{newtempeq12}
  \CA^\lam\not=\{0\}\quad\hbox{for all}\quad \lam\in\Lam.
  \end{equation}
  Further assume  that $\CA$
admits a $\Lambda$-graded  invariant nondegenerate   {symmetric}
bilinear form $\epsilon$, where ``invariant''
  means $\epsilon(ab,c)=\epsilon(a,bc)$ for all $a,b,c\in\CA$. In addition{,} we assume that
\begin{equation}\label{newtempeq6}
  \epsilon(1,1)\not=0.
\end{equation}
One gets using this that   $\epsilon(a,a^{-1})\not=0$ for all
invertible elements $a\in\CA$ as  the form is invariant. We now
consider the \la  $\fg \otimes \CA$ with multiplication defined by
  \[ [x \otimes a, y \otimes b ] = [x,y] \otimes ab\]
for every $x,y \in \fg$ and $ a,b \in \CA$.
  Now define a form on $\fg\otimes \CA$ {by linear extension} of
\begin{equation}\label{newtempeq8}
 (x \otimes a, y \otimes b ) = (x,y) \epsilon(a,b),
 \end{equation}
  for  $x,y \in \fg$ and $ a,b \in \CA$.  It is easy to see that this form is  a $\Lam$-graded invariant symmetric  bilinear form on ${\fg} \otimes \CA$.

  The following is a slight generalization of  \cite[Definition 3.1.1]{abfpgs}.
\begin{DEF}\label{loopalgebra}
The subalgebra
  \[ \tilde\fg := L_\rho(\fg, \CA):= \bigoplus_{\lambda \in \Lambda} (\fg^{\bar{\lambda}} \otimes \CA^\lambda)\]
of $\fg \otimes \CA$  is called the {\it loop algebra} of $\fg$
relative to $\rho$ and $\CA$. In the case that $\rho=0,$ we denote
$L_\rho(\fg,\CA)$ by
 $L(\fg,\CA)$ and  note that  $L(\fg,\CA)=\fg\otimes\CA$.
\end{DEF}
From definition, it is clear that $\tilde{\fg}$  is a
$\Lambda$-graded Lie algebra with homogenous spaces
$\tilde{\fg}^\lambda:=\fg^{\bar\lambda}\otimes\CA^\lambda$,
$\lambda\in\Lambda$.

In the following lemma, we make use of a fact from linear algebra,
namely if $V$ is a vector space equipped with a nondegenerate
symmetric bilinear from and $W$  a finite dimensional subspace of
$V$, then there is a finite dimensional subspace $U$ of $V$
containing $W$ such that the form restricted to $U$ is
nondegenerate (for a proof
 see \cite[Lemma 3.6]{myleala}).

\begin{lem}\label{tildefgform}
The form on $\fg \otimes \CA$ restricted to $\tilde{\fg}$ is a
$\Lam$-graded invariant nondegenerate symmetric  bilinear form.
\end{lem}

\begin{proof} As we have seen above the form on $\fg\otimes\CA$ is
$\Lam$-graded symmetric and invariant. So it remains to prove the
nondegeneracy of the form. Since $\form$ is $\Lam$-graded on
$\tilde\fg$,
 it is enough to show that for fixed $\lam\in\Lam$ and $0\not=\tilde x\in\fg^{\bar\lambda}\otimes\CA^\lam$,
 there exists $\tilde y\in \fg^{-\bar\lambda} \otimes \CA^{-\lambda}$ such that
 $(\tilde x,\tilde y)\not=0$. Now we may write  $\tilde x=\sum_{i=1}^n x_i\otimes a_i$,
 where $\{a_1,\ldots,a_n\}$ is a linearly independent subset of $\CA^\lam$ and $x_i\in\fg^{\bar \lam}$ for  $1\leq i\leq n$.
 Since $\ep$ is nondegenerate on $\CA^\lam\oplus\CA^{-\lam}$, there exists a finite dimensional subspace
 $X$ of $\CA^\lam\oplus\CA^{-\lam}$ such that $\{a_1,\ldots,a_n\}\sub X$
 and that the form restricted to $X$ is nondegenerate.
Extend $\{a_1,\ldots,a_n\}$ to a basis
$\{a_1,\ldots,a_n,a_{n+1},\ldots, a_m\}$ of $X$. Now as $\ep$ is
nondegenerate on $X$, there exist $b_1,\ldots,b_m\in X$ such that
$\ep(a_i,b_j)=\delta_{ij}$ for all $i,j$. For $1\leq j\leq n$, let
${\bar b}_j$ be the projection of $b_j$ into $\CA^{-\lam}$ with
respect to the decomposition $\CA^\lam\oplus\CA^{-\lam}$. Since
$\ep$ is $\Lam$-graded and $a_1,\ldots,a_n\in\CA^\lam$, we have
\begin{equation*}
(a_i,{\bar b}_j)=(a_i,b_j)=\delta_{i,j}\quad\hbox{for all}\quad
1\leq i,j\leq n.
\end{equation*}
 Now $x_j\not=0$ for some $j$, as $\tilde x \not=0$.
 Since $\form$ is nondegenerate and $\Z_m$-graded on
 $\fg^{\bar\lambda}\oplus\fg^{-\bar\lambda}$, there exists
 $y_j\in \fg^{-\bar\lambda}$ such that $(x_j,y_j)\neq 0$. So, setting
 $\tilde y:=y_j\otimes {\bar b}_j$, we have
\begin{eqnarray*}
  (\tilde x,\tilde y)&=&(\sum_{i = 1}^n x_i \otimes a_i,y_j \otimes {\bar b}_j)\\
  &=& \sum_{i = 1}^n (x_i,y_j)\epsilon(a_i,{\bar b}_j)\\
                                                  &=& (x_j,y_j)\epsilon(a_j,{\bar b}_j)\\
                                                  &=& (x_j,y_j) \neq 0,
\end{eqnarray*}
as required. This shows that the form on $\tilde\fg$ is
nondegenerate.
  \end{proof}
Next  suppose  $\lambda \in \Lambda,$ then by Proposition
\ref{toralpair} and (\ref{final2}), we have
\begin{equation}\label{gg}
     \tilde{\fg}^\lambda = \fg^{\bar{\lambda}} \otimes \CA^\lambda = \bigoplus_{\pi(\alpha) \in \pi(R)}
     ( \fg^{\bar \lambda}_{\pi(\alpha)} \otimes \CA^\lambda ).
\end{equation}
Now we set
 $$\tilde{T}{: =} T^0 \otimes 1.$$ Then for $\alpha \in R, \pa$  can be considered as an element
of $\left.\tilde{T}\right.^\star$
   by linear extension of $\pa(t \otimes 1) = \alpha(t)$ for  $t \in T^0$. We consider the adjoint action of
 $\tilde{T}$ on $\tilde{\fg}$. Suppose $t \in T^0, x \in \fg^{\bar{\lambda}}$ and $a \in \CA^\lambda$, for some $ \lambda \in \Lambda$. We have
   \[ [t \otimes 1, x \otimes a] = [t, x] \otimes a \in \fg^{\bar{\lambda}} \otimes \CA^\lambda. \]
   So the adjoint action of $\tilde{T}$ on $\tilde{\fg}$ stabilizes $\fg^{\bar{\lambda}} \otimes \CA^\lambda$.
   Define, for $\a\in R$,
   \[ \tilde\fg_{\pi(\alpha)} {:=} \{ x \in \tilde\fg \mid [\tilde{t},x] = \pa(\tilde{t}) x \;  \mbox{for all} \; \tilde{t}  \in \tilde T \}. \]
   Then  it is easy to check that $\fg_\pa^{\bar \lambda} \otimes \CA^\lambda \subseteq \tilde \fg_\pa$ for $\a\in R$ and $\lambda \in \Lambda$. So by
   (\ref{gg}){,}
\begin{equation}
  \begin{array}{ll}
    \tilde\fg = \bigoplus_{\lambda \in \Lambda} \tilde\fg^\lambda & = \bigoplus_{\lambda \in \Lambda}
    \bigoplus_{\pi(\alpha) \in \pi(R)} (\fg^{\bar \lambda}_{\pi(\alpha)} \otimes \CA^\lambda) \\
& = \bigoplus_{\pi(\alpha) \in \pi(R)} \bigoplus_{\lambda \in \Lambda} (\fg^{\bar \lambda}_{\pi(\alpha)} \otimes \CA^\lambda) \\
& \subseteq \bigoplus_{\pi(\alpha) \in \pi(R)} \tilde\fg_{\pi(\alpha)} \subseteq \tilde\fg. \\
  \end{array}
\end{equation}
   Thus we have
\begin{equation}\label{gtilderootdec}
  \tilde\fg = \bigoplus_{\pi(\alpha) \in \pi(R)} \tilde\fg_{\pi(\alpha)}
\end{equation}
   with
\begin{equation}\label{tildefgpa}
    {\tilde\fg}_{\pi(\alpha)} = \bigoplus_{\lambda \in \Lambda} ({\fg}^{\bar{\lambda}}_{\pi(\alpha)}\otimes\CA^\lam).
\end{equation}
    Therefore we have the following lemma.
\begin{lem} \label{QpiRgraded}
  $\tilde{\fg}$ admits a compatible $(\langle\pi(R)\rangle, \Lambda)$-grading
   \[\tilde{\fg} = \bigoplus_{\lambda \in \Lambda, \gamma \in\langle\pi(R)\rangle} \fg^{\bar \lambda}_\gamma \otimes \CA^\lambda\]
  where  for any ${\lambda \in \Lambda}$, $\tilde\fg_\gamma^\lambda =\{0\}$ if $\gamma \notin \pi(R)$.
\end{lem}
   Consider the $\F$-vector space
\begin{equation}\label{newtempeq17}
\CV:=\F \otimes_\Z \Lambda.
\end{equation}
Since $\Lambda$ is torsion free, we may identify
   $\Lambda$ with the subgroup $1\otimes\Lambda$ of $\CV$. Now as $\Lambda$
   spans $\CV$, it contains a basis $\{\lambda_i\mid i\in I\}$ of $\CV$.
For any $i \in I$, set $d_i\in \CV^\star$ by
$d_i(\lambda_j):=\delta_{ij}$, $j\in I$, and let $\CV^\dag$ be the
restricted dual of $\CV$ with respect to the basis $\{\lam_i\mid
i\in I\}$, namely
\begin{equation}\label{newtempeq18}
\CV^\dag:=\hbox{span}_\F\left\{ d_i\mid i \in I\right\} \subseteq
\CV^\star.
\end{equation}
 Define
\begin{equation}\label{hatfg}
  \hat \fg:=\widehat{L_\rho(\fg,\CA)}:=\tilde \fg\oplus \CV \oplus \CV^\dag
  \andd \hat T:=\tilde T\oplus \CV\oplus \CV^\dag=(T^0\otimes 1)\oplus\CV\oplus\CV^\dag.
\end{equation}
If $\rho=0$, we denote  $\hat\fg$  by $\widehat{L(\fg,\CA)}$. We
make $\hat \fg$ into  a \la by letting the Lie bracket  be
\begin{equation}\label{bracketloop}\begin{array}{l}
  \;[d,x]=d(\lambda) x,  \quad d\in\CV^\dag,x\in \tilde \fg^\lambda,\lambda \in \Lambda,\\
 \;[\CV,\hat\fg]=\{0\},\\
\;[x,y]=[x,y]_{\tilde\fg}+\sum_{i\in I}([d_i,x],y)\lambda_i, \quad
x,y\in \tilde\fg, \end{array}\end{equation} where by
$[\cdot,\cdot]_{\tilde\fg}$ and $(\cdot,\cdot)$, we mean the Lie
bracket and the  bilinear form on ${\tilde\fg},$ respectively.
Note that {for each $x,y\in\tilde\fg$, $\sum_{i\in
I}([d_i,x],y)\lambda_i$ makes sense as $[d_i,x]=0$, for all but a
finite number of $i \in I$.} We next extend the form on
$\tilde\fg$
  to a bilinear form on $\hat \fg$ by
\begin{equation}\label{formloop}
\begin{array}{l}
(\CV,  \CV)=(\CV^\dag,\CV^\dag)=(\CV,\tilde\fg)=(\CV^\dag,\tilde \fg):=\{0\},\\
(v,d)=(d,v):=d(v),\quad  d\in \CV^\dag, v\in \CV.
\end{array}
\end{equation}
  The above form is clearly nondegenerate on $\hat\fg$. For any $\lambda\in \Lambda$,
  define $\delta_\lambda \in \left.\hat{T}\right.^\star$ by $$\delta_\lambda((T\otimes 1)\oplus\CV)=\{0\}, \delta_\lambda(d)=(\lambda, d),\;\;d\in \CV^\dag.$$
  {Then the assignment $\lambda \mapsto \delta_\lambda$ affords an embedding of $\Lam$ into $\left.\hat{T}\right.^\star$, by the nondegeneracy of
  $\form$.}
  So we may identify $\lambda$ with $\delta_\lambda$ and suppose that $\Lambda \subseteq \left.\hat{T}\right.^\star$.

For $\a\in R,$ one can extend  $\pa\in \pi(R)$ to $\left. \hat
T\right.^\star$
  by defining $\pa(\CV\oplus\CV^\dag):=\{0\}$.
  Now let $x\in \tilde \fg_\pa^\lambda,$ $\lam\in\Lam,\a\in R,$ and $(t\otimes1)+v+\bar v\in \hat T$, $t\in T^0,v\in \CV,\bar v\in\CV^\dag,$ then
\begin{eqnarray*}
  [(t\otimes 1) + v +\bar v,x]&=& [t \otimes 1,x]_{\tilde\fg} + \bar v(\lambda)x\\
                            &=& (\pa(t\otimes 1) + \bar v(\lambda))x\\
                            &=& (\pa+\lambda)((t\otimes 1)+ v +\bar v) x.
\end{eqnarray*}
This shows that
\begin{equation}\label{newtempeq13}
\hat\fg=\bigoplus_{\tilde\alpha\in \left.{\hat
T}\right.^\star}{\hat\fg}_{\tilde\alpha},
\end{equation}
where
$${\hat\fg}_{\tilde\alpha}:=\{x\in\hat\fg\mid [\hat t,x]=\tilde\alpha(\hat t)x\hbox{ for all }\hat t\in \hat T\}.
$$
That is, $(\hat\fg,\hat T)$ is a toral pair. Moreover, if $\hat R$
is the root system of $(\hat\fg, \hat T)$, then
\begin{equation}\label{newtemeq8}
\hat R\subseteq \pi(R)\oplus\Lambda,
\end{equation}
{and for $\a\in R$ and $\lambda\in\Lam$,
\begin{equation}\label{hatfgpa+lambda}
\hat\fg_{\pa+\lambda}=\left\{\begin{array}{ll}
\fg_\pa^{\bar\lambda}\otimes\CA^\lam&\mbox{if }\pa+\lambda\not=0,\\
(\fg^0_{\pi(0)} \otimes \CA^0 )\oplus \CV \oplus \CV^\dag&\hbox{if
}\pa+\lambda=0.
\end{array}\right.
\end{equation}
} Next for $\lam\in\Lam$,  we put
\begin{equation}\label{newtempeq15}
 R_{\bar\lambda}:=\{ \alpha\in R\mid \fg_\pa^{\bar\lambda}\neq \{0\}\},
 \end{equation}
  then it follows from Lemma \ref{QpiR,Hgraded}, (\ref{newtempeq13}),
  (\ref{hatfgpa+lambda}), (\ref{gtilderootdec}), (\ref{tildefgpa}) and (\ref{newtempeq12}) that
\begin{equation}\label{newtempeq11}
R=\bigcup_{\lam\in\Lam}R_{\bar\lam}\andd\hat R=
\bigcup_{\lambda\in\Lambda} (\pi(R_{\bar\lambda}) + \lambda).
\end{equation}

Now we can prove the main theorem of this section which is a
rather comprehensive extension of \cite[Theorem
3.63]{abpcovering1}.

\begin{thm}\label{Lrhofgisiara}
  {Let $(\fg, T)$ be a division IARA with corresponding root system $R$. Suppose $\sigma$
  is an automorphism of $\fg$ satisfying (A1)-(A4). {Assume further that either (A5) holds or $\fg_0$ is
  abelian.}
  Suppose $\Lambda$ is a torsion free abelian group and $\rho:\Lambda\lra \Z_m$ a group epimorphism.
  In addition, let $\CA$ be a unital commutative associative predivision $\Lambda$-graded algebra,
  with $\supp_\Lam(\CA)=\Lam$.
  Then $(\hat\fg=\widehat{L_\rho(\fg,\CA)},\hat{T})$ is an IARA with root system $\hat R=\cup_{\lam\in\Lam}(\pi(R_{\bar \lam})+\lam)$.
  Moreover, if $R$ is indecomposable then so is $\hat R$. Finally,
  if $T$ is a splitting Cartan subalgebra of $\fg$
  and $\CA^0=\F$, then $\hat T$ is also a splitting {Cartan subalgebra of }
  $\hat\fg$.}
\end{thm}

\begin{proof}
  We have already seen that $(\hat\fg,\hat T)$ is a toral pair, so it remains to verify conditions
  (IA1)-(IA3) of Section \ref{pre}.
  We know that the form introduced by (\ref{formloop}) on $\hat \fg$
  is nondegenerate  on both $\hat\fg$ and $\hat T$
  and so (IA1) holds for $\hat\fg$.

 {We next show that (IA2) holds. Assume that $\a\in R$, $\lambda\in\Lam$,
 $\pa+ \lambda \neq 0$ and $\hat\fg_{\pa+\lambda} \neq 0$. By (\ref{hatfgpa+lambda}),
  $\hat\fg_{\pa+\lambda} = \fg_\pa^{\bar\lambda}\otimes\CA^\lambda$,
  so $\fg_\pa^{\bar\lambda} \neq 0$ and $\CA^\lambda \neq 0$. As $\CA$ is predivision $\Lam$-graded,
  there exists $a\in \CA^\lambda$ and $b\in \CA^{-\lambda}$ such that $ab=1$.
To proceed with the proof, we  divide the argument into two cases
$\pi(\a)\not=0$ and $\pi(\a)=0$.
  Assume first that $\pa\neq 0$, then by Lemma~\ref{star}, there exist $0 \neq x \in \fg_\pa^{\bar \lambda}$ and
  $ 0\neq y \in \fg_{-\pa}^{-\bar\lambda}$ such that $0\neq [x,y]\in T^0$, and  thus
  \[ [x\otimes a,y\otimes b]=([x,y]\otimes 1)+ \sum_{i\in I}([d_i,x\otimes a],y\otimes b)\lambda_i \in \hat
  T\setminus\{0\},\]} as required.

{Next, assume $\pa=0,$ then by Lemma~\ref{afterstar}, there exist
$x\in \fg_{\pi(0)}^{\bar\lam}$ and $y \in
\fg_{\pi(0)}^{-\bar\lam}$ such that $[x,y]=0$ but $(x,y)\neq 0$.}
{So we have
\begin{align*}
[x\otimes a,y \otimes b]&= ([x,y] \otimes 1) +\sum_{i\in I} ([d_i,x\otimes a],y\otimes b)\lam_i\\
                          &= 0 +\sum_{i\in I}\epsilon(a,b)d_i(\lambda)(x,y)\lambda_i.
\end{align*}
This is a nonzero element of $\hat T$ since $(x,y)\neq 0$,
$\epsilon(a,b)= {\epsilon(1,1)} \not=0$ and as
$\lam=0+\lam=\pi(\a)+\lam\neq 0,$ $d_i(\lambda)\not=0$ for some
$i\in I$. This means that  (IA2) holds for $\hat \fg$.}

  {Finally, we consider (IA3). Let $\a\in R$, $\lambda\in\Lambda$ and $(\pi(\a)+\lambda,\pi(\a)+\lambda)\not=0$.
  {A}s $(\lam,\lam)=(\lam,\pi(\a))=0$, we have $(\pi(\a),\pi(\a))\not=0$.
  Since by Theorem~\ref{gt0isiara}, $\pi(R)$ is an affine reflection system, one
  can use a similar technique as in the proof of Theorem \ref{gt0isiara} to show that $\ad(x)$
  is locally nilpotent for any $x \in \hat\fg_{\pa+\lambda}$. So $\hat\fg$ satisfies (IA3) and $\hat\fg$ is an IARA. Moreover,
  the root system  $\hat R$ of $(\hat\fg,\hat T)$ satisfies $\hat R={\bigcup}_{\lam\in\Lam}(\pi(R_{\bar\lam})+\lam)$, by
  (\ref{newtempeq11}).}

 Next, suppose $R$ is indecomposable. Since $\Lam$ is contained in the radical of the form, $\hat R$
is indecomposable if and only if
$\cup_{\lam\in\Lam}\pi(R_{\bar\lam})$ is indecomposable.
  But by (\ref{newtempeq11}) this union is $\pi(R)$ which is indecomposable by Lemma \ref{newtemplem6}.

To see the final assertion {of the theorem}, we note that if
$\fg_0=T,$ then   by (A4), $C_{\fg^0}(T^0)=T^0.$ Therefore as
$\CA^0=\mathbb{F},$ using  (\ref{hatfgpa+lambda}), we have
\begin{eqnarray*}
\hat\fg_0&=&\big((\fg_{\pi(0)}\cap\fg^{0})\otimes\CA^0\big)\oplus\CV\oplus\CV^\dag\\
&=&
\big(C_{\fg^0}(T^0)\otimes 1\big) \oplus\CV\oplus\CV^\dag\\
&=&(T^0\otimes 1)\oplus\CV\oplus\CV^\dag=\hat T.
\end{eqnarray*}
Thus $\hat T$ is a splitting Cartan subalgebra of $\hat\fg$ as
required.
\end{proof}

\begin{cor}\label{newtempcor1}
  Let $(\fg, T)$ be a{n} IARA with corresponding root system $R$ and bilinear from $\form$.
  Let $\Lambda$ be a torsion free abelian group and $\CA$ be a unital commutative associative predivision $\Lambda$-graded algebra,
  with $\supp_\Lam(\CA)=\Lam$. Define $\hat\fg:=(\fg\otimes\CA)\oplus\CV\oplus\CV^\dag$ and $\hat T:=(T\otimes 1)\oplus\CV\oplus\CV^\dag$,
  where $\CV$ and $\CV^\dag$ are defined as
in (\ref{newtempeq17}) and (\ref{newtempeq18}), respectively. Then
$(\hat\fg,\hat{T})$ is a{n} IARA with root system $\hat
R=R\oplus\Lam$. Moreover, if $R$ is indecomposable then so is
$\hat R$.
\end{cor}

\begin{proof}
Taking $\sg$ to be the identity automorphism and recalling from
Remark \ref{temprem1} that $T\neq\{0\}$, it is apparent that $\sg$
satisfies conditions (A1)-(A5). Therefore, if $(\fg,T)$ is
division, we are done by Theorem \ref{Lrhofgisiara}. Now  a close
look at the proof of Theorem \ref{Lrhofgisiara}, shows that the
division property, that is (IA2)$'$, guarantees the existence of
nonzero elements $x\in \fg^{\bar \lam}_{\pi(\a)}$ and
$y\in\fg^{-\bar \lam}_{-\pi(\a)}$ ($\a\in R$ and
$\lambda\in\Lambda$ with $\pi(\a)+\lam\neq0$) such that
\begin{equation}\left\{\begin{array}{ll}[x,y]\in
T^0\setminus\{0\}&
\hbox{if $\pi(\a)\neq0$}\\
{[x,y]=0\andd (x,y)\neq0}& \hbox{if
$\pi(\a)=0.$}\end{array}\right.\label{final}\end{equation}
 However
when $\sg$ is the identity automorphism, (\ref{final}) clearly
holds with the weaker axiom (IA2). Finally, since $\sg$ is the
identity automorphism, it follows immediately from Theorem
\ref{Lrhofgisiara} that $\hat R=R\oplus\Lam$
\end{proof}

\begin{cor}\label{division case}
  {Let $(\fg, T)$ be a division IARA with corresponding root system $R$. Suppose $\sigma$
  is an automorphism of $\fg$ satisfying (A1)-(A4). Assume
  further that 
  {$\fg_0$ is abelian.}
  Suppose $\Lambda$ is a torsion free abelian group and $\rho:\Lambda\lra \Z_m$ a group epimorphism.
  In addition, let $\CA$ be a {commutative} associative $\Lambda$-torus,
  with $\supp_\Lam(\CA)=\Lam$.
  Then $(\hat\fg=\widehat{L_\rho(\fg,\CA)},\hat{T})$ is a division
  IARA, with root system $\hat{R}$.}
\end{cor}

\begin{proof}
By Theorem~\ref{Lrhofgisiara}, $(\hat\fg,\hat T)$ is an IARA. So
the only condition which we should verify is (IA2)$'$.
  Suppose $\a\in R$, $\lam\in\Lam$, $\pi(\a)+\lam\not=0$ and $\hat\fg_{\pa+\lambda}=\fg_\pa^{\bar\lambda} \otimes
  \CA^\lambda \neq \{0\}$. Since $\CA$ is a $\Lam$-torus, $\CA^\lambda$ is one dimensional, say $\CA^\lambda = \hbox{span}_\F\{a\}$,
  where $a$ is invertible with inverse $b$.
  Then any element of $\hat\fg_{\pa+\lambda}$ is of the form $x \otimes a$ for some $0\not=x \in \fg_\pa^{\bar\lambda}$.
 Now fix a nonzero element $x\otimes a\in\hat\fg_{\pa+\lambda}.$ If $\pa \neq 0$, then by Lemma~\ref{star},
 there exists $y \in \fg_{-\pa}^{-\bar\lambda}$ such
  that $0 \neq [x,y]\in T^0.$ So as $[x,y]\not=0$, we have
 \[ [x\otimes a,y\otimes b]=([x,y]\otimes 1)+ \sum_{i\in I}([d_i,x],y)\lambda_i \in \hat T\setminus\{0\}.\]

Now suppose  $\pa =0.$   We claim that there exists  $y \in
\fg_{\pi(0)}^{-\bar\lam}$ such that $[x,y]=0$  and $(x,y) \neq 0.$
For this, take $j\in\Z$ such that $\bar j=\bar\lam$. By Lemma
\ref{templem2}(ii), we have
$$\fg^{\bar\lam}_{\pi(0)}=\fg^{\bar j}_{\pi(0)}=\sum_{\{\b\in\hbox{orb}(R)\mid\pi(\b)=0\}}\pi_j(\fg_\b).$$
Therefore, $x = \pi_j(x_1) + \cdots + \pi_j(x_n)$ where $x_i \in
\fg_{\a_i}$ for some $\a_1, \ldots, \a_n$ belong to distinct
$\sg$-orbits of $R$ and $\pi(\a_i)=0, 1\leq i \leq n$. As $x \neq
0$, $\pi_j(x_k)\neq 0$ for some $\F$. Now if $\a_k\not=0$, then by
Lemma~\ref{(pij,pi-j)neq0}, there exists $y_k\in\fg_{-\a_k}$ such
that $(\pi_j(x_k),\pi_{-j}(y_k))\not=0$ and
$[\pi_j(x_k),\pi_{-j}(y_k)]=0$. Set $y:=\pi_{-j}(y_k),$ then by
\rred{Lemma~\ref{[pi_j,pi_-j]}~(iii)}, we have
$[x,y]=[\pi_j(x_k),\pi_{-j}(y_k)]=0$. Also by
Lemma~\ref{[],()}(ii), we have
\begin{align*}
(x,y)=\sum_{i=1}^n(\pi_j(x_i), \pi_{-j}(y_k)) &= \frac{1}{m}
\sum_{i=1}^n( \overline{(x_i)}_j, y_k)
=(\pi_j(x_k),\pi_{-j}(y_k))\not=0,
\end{align*}
where considering (\ref{xj}), we note that for $i\not=k$,
$(\overline{({x}_i)}_j,y_k)\in (\fg_{\a_i},\fg_{-\a_k})=\{0\}$, as
$\a_i-\a_k\not=0$. So we are done in the case $\a_k\not=0$.

Next, suppose that $\a_k=0$. Then
$0\not=\pi_j(x_k)\in\pi_j(\fg_0)\sub\fg^{\bar j}_{\pi(0)}$. Since
$\form$ is nondegenerate on $\fg_0$, there exists $ y \in \fg_0$
such that $(\pi_j(x_k),y)\neq 0$. But as $\form$ is $\Z_m$-graded,
we may assume that $y=\pi_{-j}(y)\in\fg_{0}^{-\bar
j}\sub\fg^{-\bar j}_{\pi(0)}$. Then $[\pi_j(x),y]=0$, as by
assumption $\fg_0$ is abelian. Now repeating the same argument as
in the case $\a_k\not=0$ (using
\rred{Lemmas~\ref{[pi_j,pi_-j]}~(iii)} and \ref{[],()}), we get
$(x,y)\not=0$ and $[x,y]=0$. This completes the proof of the
claim.  Now we note that $\lam\not=0$ as $\pi(\a)+\lam\not=0$. So
$d_j(\lam)\not=0$ for some $j\in I$. Therefore we have
\begin{align*}
 [x\otimes a,y \otimes b]&= 0 +\sum_{i\in I} ([d_i,x\otimes a],y\otimes b)\lam_i=\sum_{i\in I}d_i(\lam)\epsilon(a,b)(x,y)\lam_i \in \hat T,
\end{align*}
 which is nonzero, as $\epsilon(a,b)\not=0$, $(x,y)\not=0$ and $d_j(\lam)\not=0.$
\end{proof}

Suppose that $(\fg,T)$ is an IARA with root system $R$ and
$\sigma$ is an automorphism of $\fg$ satisfying (A1)-(A4) such
that the order of $\sg$ is prime. As we have already seen,  the
automorphism $\sigma$ induces a linear isomorphism
$\sg:T^\star\longrightarrow T^\star$ with $\sigma(R)=R.$ In fact
$\sg$ is an automorphism of $R$ in the sense of
\cite{nehersurvey}. The following lemma shows that
$\sigma(\delta)=\delta$ for each $\delta\in R^0.$ In particular,
one gets that  an automorphism of an IARA, satisfying the above
conditions, preserves each isotropic  root space. This is a
nontrivial fact that one should consider in constructing suitable
automorphisms of IARA's.

\begin{lem}\label{newlem-1}
  {Suppose  $(A, \form, R)$ is a tame affine reflection system. In addition, suppose that $A$ is 2-torsion free, and
$\sigma$ is an automorphism of $A$ with $\sigma(R)=R$ (a root
system automorphism) of period $m$ such that
$\pi(\delta):=(1/m)\sum_{i=1}^{m-1}\sg^i(\delta) \neq 0$ for any
nonzero $\delta \in R^0$.
  Then $\sigma(\delta)=\delta$ for each $\delta\in R^0$.}
\end{lem}
\begin{proof}
Since $R$ is tame, it follows from \cite[Theorem 1.13]{AYY} that
\begin{equation}\label{neweq-9}
R^0+2\langle R^0\rangle\sub R^0.
\end{equation}
Now suppose $\delta \in R^0$. Then by (\ref{neweq-9}), $n\delta\in
R^0$ for
  all $n\in\Z$ and so $n\sg(\delta)\in R$ for all $n$. But this can happen only if $\sg(\delta)\in R^0$ \cite[Theorem 1.13]{AYY}.
Now again from (\ref{neweq-9}), we have
  $2\delta-\sigma(2\delta) \in R^0$.
  But $\pi(2\delta-\sigma(2\delta))=0$ and so by assumption $2(\delta-\sigma(\delta))=0$. Now since $A$ is 2-torsion free,
  we get $\sg(\delta)=\delta$ as required.
\end{proof}

\begin{rem}\label{newtemprem2}
{In this remark, we discuss the structure of a commutative
associative predivision $\Lam$-graded algebra $\CA$, $\Lam$ an
abelian group. We refer the reader to \cite[Section
4.5]{nehersurvey} for a more general discussion. As we have
already mentioned, $\supp_\Lambda(\CA)$ is a subgroup of
$\Lambda$, and so without loss of generality we may suppose that
supp$_\Lambda(\CA) = \Lambda$. Suppose $\left\{ u_\lambda \mid
\lambda \in \Lambda \right\}$ is a family of invertible elements
$u_\lambda \in \CA^\lambda$. Put $B := \CA^0$, then
$\CA^\lam=Bu_\lam$ for all $\lam$ and $\{ u_\lambda
\}_{\lam\in\Lam}$ is a free basis for the $B$-module $\CA$ and the
{multiplication on $\CA$} is uniquely determined by
\begin{equation}\label{product}
  u_\lambda u_\mu = \tau(\lambda,\mu) u_{\lambda+\mu} \qquad \mbox{and} \qquad u_\lambda b = b u_\lambda \quad (b \in B),
\end{equation}
  where $\tau : \Lambda \times \Lambda \lra U(B)$ is a function, $U(B)$ being the group of units of $B$. Associativity and commutativity
of $\CA$ leads to
\begin{equation}\label{tau}
\tau (\lambda,\mu) \tau(\lambda+\mu,\nu) = \tau(\mu,\nu)
\tau(\lambda,\mu+\nu), \qquad \tau(\lambda,\mu) =
\tau(\mu,\lambda),
\end{equation}
  for $\lambda,\mu,\nu \in \Lambda$. In other words,
  $\tau:\Lam\times\Lam\longrightarrow U(B)$ is a {\it symmetric
  $2$-cocycle}.
  Conversely, given any unital commutative associative $\F$-algebra $B$ and a symmetric $2$-cocycle $\tau:\Lam\times\Lam\longrightarrow U(B)$,
one can define a commutative associative predivision
$\Lam$-graded $\F$-algebra by (\ref{product}). To be more
precise, let $\CA$ be the free $B$-module with basis
$\{u_\lam\}_{\lam\in\Lam}$, namely
$\CA:=\bigoplus_{\lam\in\Lam}Bu_\lam$. Then, identifying $B$
with $Bu_0$ through $b\longmapsto b\tau(0,0)^{-1}u_0$, $b\in B$, and using
(\ref{product}) as the multiplication rule on $\CA$, we get the
desired algebra.} A commutative associative algebra arising in
this way is called a {\it twisted group algebra}
  and is denoted by $B^t[\Lambda]$.
  To summarize, any commutative associative predivision graded algebra $\CA$ with support $\Lambda$ is
  graded isomorphic to a twisted group algebra $B^t[\Lambda]$.
  {It follows that},
  $\CA$ is division graded if and only if $B$ is a field{,} and is an associative $\Lam$-torus if and
  only if $B=\F$.
\end{rem}

\section{\bf Examples}\setcounter{equation}{0}\label{examples}
In this  section, we illustrate extended affinization through some
examples. In the first example, using extended affinization
process,  we construct  a generalization of the class of
\textit{toroidal} Lie algebras. In the second example, starting
from a certain IARA, we show that we can iterate extended
affinization process to get a series of IARA's. Finally, in the
last example, we apply extended affinization  starting from an
IARA of type $A$ and ending up with an IARA of type $BC.$ Before
going to the main body of this section, we make a convention that
in each example,
\begin{equation}\tag{{\Large$\star$}}\label{B}
\parbox{4in}{ $B$ is a unital
associative algebra over $k$
  admitting an invariant nondegenerate  symmetric bilinear form $\epsilon$ such that $\epsilon(1,1)\neq
  0$.}
  \end{equation}

\begin{ex}\label{firstex}
Suppose $(\fg, T)$ is an IARA with corresponding root system $R$
and bilinear form $\form$. Assume $B$ is commutative and let
$\Lambda$ be a torsion free abelian group. Consider the twisted
group algebra $B^t[\Lambda]: =\bigoplus_{\lambda\in
\Lambda}Bz^\lambda$ and recall that we have $bz^\lambda
cz^\mu=bc\tau(\lam,\mu)z^{\lambda+\mu}$, $z^\lambda b =
bz^\lambda$ for $b,c\in B,\lambda,\mu\in \Lambda$ where $\tau:\Lam
\times \Lam:\lra U(B)$ is a symmetric $2$-cocycle. We extend
$\epsilon$ to {$B^t[\Lam]$} by linear extension of
\begin{equation}\label{epsilon}
 \epsilon(bz^\lambda,cz^\mu):=\left\{
  \begin{array}{ll}
    \epsilon(b,c)& \lambda+\mu=0\\
    0&\lambda+\mu\neq0.
  \end{array}
  \right.
\end{equation}
{Set $B^t[\Lam]^\lam:=Bz^\lam$, $\lam\in\Lam$.} Then by
Remark~\ref{newtemprem2}, {$B^t[\Lam]$} is a commutative
associative predivision $\Lam$-graded algebra over $\F$ and one
can easily verify that
  $\epsilon$ is a $\Lambda$-graded invariant nondegenerate symmetric  bilinear form on {$B^t[\Lam]$}.

Define $\hat\fg$ and $\hat T$ as in (\ref{hatfg}) with  $\rho=0.$
Namely $$\hat\fg=\widehat{L(\fg,B^t[\Lam])}=(\fg\otimes
B^t[\Lam])\oplus\CV\oplus\CV^\dag\andd\hat T=(T\otimes
1)\oplus\CV\oplus\CV^\dag,$$ with corresponding Lie bracket and
bilinear form defined by (\ref{bracketloop}) and (\ref{formloop}),
respectively. Then by Corollary \ref{newtempcor1}, $(\hat\fg,\hat
T)$ is an IARA with root system $\hat R=R\oplus\Lam$. We note that
this structure in fact generalizes the well known structure of
{\it toroidal Lie algebras}.
\end{ex}

\begin{ex}\label{newtempexa2}
We continue with the same notations as in  Example \ref{firstex},
in particular $\hat\fg=\widehat{L(\fg,B[\Lam])}.$ Set
$\CA:=B[\Lam]$ and suppose $\sigma\in \Aut(\fg)$ satisfies axioms
(A1)-(A4).
  Let $\mu:\Lambda\lra\Z$ be a group homomorphism. The map $\mu$ induces an automorphism of $\CA,$ denoted again by $\mu,$ defined by   $\mu(x):=\zeta^{\mu(\lambda)}x$ for any $x \in \CA^\lambda$, where $\zeta$ is a primitive $m$-th root of unity.
  Both $\sigma$ and $\mu$ can be considered as automorphisms of $\hat\fg$ by
$$\begin{array}{c}
    \sigma=\sg\otimes\id\hbox{ on }\fg\otimes\CA\hbox{ and }\sg=\id\hbox{ on }\CV\oplus\CV^\dag,\\
\mu=\id\otimes\mu\hbox{ on }\fg\otimes\CA\hbox{ and }\mu=\id\hbox{
on }\CV\oplus\CV^\dag.
\end{array}
$$
Set $\hat\sigma:=\sigma \mu \in \Aut(\hat\fg)$. We claim that
$\hat\sigma$ satisfies (A1)-(A4). Since $\sigma$ and $\mu$ commute
and both are of period $m$ over $\hat\fg$, (A1) holds. Also  (A2)
holds since $\sigma$ and $\mu$ stabilize $T\otimes1$, $\CV$ as
well as  $\CV^\dag,$ and  (A3) holds since $\sigma$ preserves the
form $\form$ on $\fg$ and $\mu$ preserves the form  $\epsilon$ on
{$\CA$}. For (A4), first note that
  \[ {\hat\fg}^0 =
  \big( \sum_{\lambda\in\Lambda}{\fg}^{-\overline{\mu(\lambda)}}
  \otimes \CA^\lambda\big) \oplus \CV\oplus\CV^\dag
  \quad\mbox{and}\quad \hat{T}^0=(T^0\otimes 1) \oplus \CV\oplus\CV^\dag.\]
Also
  \[ C_{\hat\fg}({\hat T}^0)=(C_{\fg}(T^0) \otimes \CA^0) \oplus \CV\oplus\CV^\dag, \]
  so
  \[C_{{\hat\fg}^0}(\hat{T}^0)=(C_{\fg^0}(T^0)\otimes \CA^0)\oplus \CV\oplus\CV^\dag.\]
  Since $\sigma$ satisfies (A4), $C_{\fg^0}(T^0)\subseteq\fg_0$. Thus
  \[ C_{{\hat\fg}^0}(\hat{T}^0)\subseteq (\fg_0\otimes \CA^0)\oplus \CV\oplus\CV^\dag ={\hat\fg}_0\]
  and (A4) holds for $\hat\sigma$.

{We now further assume that $\fg$ is division, $\fg_0$ is abelian
and $B=\F$. Then by Remark~\ref{newtemprem2} and Corollary
\ref{division case}, $(\hat\fg,\hat T)$ is a division IARA. In
addition, $\hat\fg_0=(\fg_0\otimes B) \oplus\CV \oplus \CV^\dag$
is abelian.} Therefore $(\hat\fg,\hat T)$ and $\hat\sg$ satisfy
conditions of Theorem \ref{Lrhofgisiara}, with $\hat\fg$, $\hat T$
and $\hat\sg$ in place of $\fg$, $T$ and $\sg$, respectively. Now let $\Lam'$ be a torsion-free abelian group, $\rho':\Lam' \lra \Z_m$
a group epimorphism and $\CA'$ a suitable $\Lam'$-graded commutative associative algebra. Then starting from $(\hat\fg,\hat T)$ and $\hat\sigma$, one can use
Theorem \ref{Lrhofgisiara} to construct a new IARA
$\widehat{L_{\rho'}(\hat\fg,\CA')}$. This process can be iterated using suitable inputs.
\end{ex}

\begin{ex}
  Suppose $J$ is a nonempty  index set, with a fixed total ordering, and $\mathbf{q}=(q_{ij})$ is a $J\times J$ matrix over $k$
  such that $q_{ij}=\pm1$, $q_{ji}=q_{ij}$ and $q_{ii}=1$, for
  all $i,j \in {J}$. We recall  that  $B$ and $\epsilon$
are as in (\ref{B}). Let $\CA
:=B_{\mathbf{q}}[z_j^{\pm1}]_{j\in J}$ be the unital associative
algebra
  generated by
  $\{ z_j, z_j^{-1}, b \mid  j\in J, b\in B\}$ subject to the relations
\begin{equation}\label{relations}
  z_jz_j^{-1}=z_j^{-1}z_j=1,\quad z_iz_j=q_{ij}z_jz_i \quad\mbox{and}\quad z_ib=bz_i,\quad (i,j \in J,\; b\in B).
\end{equation}
Take $\Lam:=\Z^{|J|}$ and for $\lam=(\lam_j)_{j\in J}\in\Lam.$
Set $z^\lam:=\Pi_{j\in J} z_j^{\lambda_j}\in\CA$, where product
makes sense with respect to the total ordering on $J$. Then
$\CA$ is a predivision $\Lambda$-graded associative algebra
with $\CA^\lambda=B z^\lambda$ for each $\lambda\in\Lambda$.
Moreover, a similar argument as in \cite[Proposition 2.44]{bgk}
shows that \begin{equation}\label{center} \CA=[\CA,\CA] \oplus
Z(\CA).
\end{equation}

Let $K$ be a nonempty index set and denote by  $K^\pm$, the set
$K\uplus \{0\} \uplus (-K)$ where $-K$ is a copy of $K$ whose
elements are denoted by $-k, k\in K$.  Let $\CK$ be the Lie
subalgebra $\mathfrak{sl}_{K^\pm}(\CA)$ of all finitary
$K^\pm\times K^\pm$ matrices over $\CA$ generated by the
elementary matrices $ae_{ij}$, $ i\neq j\in K^\pm$, $a\in \CA$
(for details the reader is referred to \cite[Section 7]{nehersurvey}).
One knows that there is a unique $\Lam$-grading on $\CK$ such that
for each $ i \neq j \in K^\pm$ and $a\in \CA^\lambda,$
 $ae_{ij}\in\CK^\lambda.$

  One can extend $\epsilon$ {from $B$ to $\CA$ as in} (\ref{epsilon}), then we can define a
  $\Lam$-graded invariant  nondegenerate symmetric  bilinear form on
  the set of finitary $K^\pm\times K^\pm$-matrices
  by linear extension of
  $$(ae_{ij},be_{ks})_\CK:=\delta_{i,s}\delta_{j,k}\epsilon(a,b)\quad\hbox{for}\quad a,b\in \CA,  i,j,k,s \in
  K^\pm.$$
By  \cite[Section~7.10]{nehersurvey} the restriction of this form
to $\CK$ is nondegenerate if and only if $Z(\CK)=\{0\}$. Also by
\cite[Section 7.4]{nehersurvey}, $Z(\CK)=\{0\}$ if $|K|=\infty$,
and
\begin{equation}\label{zck}
  Z(\CK)= \{ z I_{2n+1} \mid z \in Z(\CA), (2n+1) z \in [\CA,\CA]\},
  \end{equation}
  if $|K|=n<\infty$,
   where by $I_{2n+1}$ we mean the
identity $(2n+1) \times (2n+1)$ matrix. Therefor by
(\ref{center}), $Z(\CK)=\{0\}$ in this case too. So the
restriction of the form to $\CK$ is $\Lam$-graded and
nondegenerate.

Suppose $\dot T:=\hbox{span}_\F\{e_{ii}-e_{jj} \mid i\neq j \in
K^\pm\}$. For $i\in K^\pm$, define
 $\varepsilon_i\in \dot{T}^{^\star}$ by $\ve_i(e_{jj}-e_{kk}):=\delta_{ij}-\delta_{ik}$, for $j\not=k\in K^\pm$.
 For $i,j\in K^\pm$, put $\dot\alpha_{ij}: = \ve_i-\ve_j$ and
 $\dot R :=\{ \dot\alpha_{ij} \mid i,j \in K^\pm\}$. For $\dot\a\in\dot R$, set
 $\CK_{\dot\a}:=\{x\in\CK\mid [t,x]=\dot\a(t)x\hbox{ for all }t\in\dot T\}.$ Then
  $\CK=\bigoplus_{\dot\a\in \dot R} \CK_{\dot\a}$. We note that $\CK_0$  consists of diagonal elements of $\CK$.
  In addition, if we assume that $\ep(1,1)=1$, then for any $i,j,s,k \in \CK$ with $i\neq j$ and $s\neq k$,
we have {\begin{eqnarray*}
  (e_{ii}-e_{jj},  e_{kk} - e_{ss})_\CK&=& \delta_{ik}-\delta_{is}-(\delta_{jk}-\delta_{js})\\
                                   &=& \epsilon_i(e_{kk}-e_{ss})-\epsilon_j(e_{kk}-e_{ss})\\
                                   &=& \dot\alpha_{ij}(e_{kk}-e_{ss}).
\end{eqnarray*}}
  {Thus $t_{{\dot\a}_{ij}}:=e_{ii}-e_{jj}$ is the unique element in $\dot T$ representing ${\dot\alpha}_{ij}$ via
  $\form_\CK$.}

Next, consider the $\F$-vector space  $\CV:=\F\otimes_\Z\Lam,$
identify $\Lam$ as  a subset of $\CV$ and fix a basis
$\{\lam_j\mid j\in J\}$ for $\CV.$ Define  the vector space
$\CV^\dag:=\sum_{j\in J}\F d_j\subseteq {\CV}^{\star}$ as in
(\ref{newtempeq18}). Set
 \[\fg:=\CK\oplus\CV\oplus\CV^\dag \andd T:=\dot T\oplus\CV\oplus \CV^\dag.\]

Define the Lie bracket on $\fg$ as in (\ref{bracketloop}), and
extend the form $\form_\CK$ on $\CK$ to a form $\form$ on $\fg$ as
in (\ref{formloop}). Then it is clear that $\form$ is
nondegenerate both on $\fg$ and $T$.

We note that each $\dot\alpha \in \dot R$ can be considered as an
element of $T^{^\star}$ by requiring
$\dot\alpha(\CV)=\dot\alpha(\CV^\dag):=\{0\}$.
 One can easily see that $t_{\dot\alpha}$ represents
  $\dot \alpha$ via $\form$ for each  $\dot\alpha \in \dot R$.
  Also we can consider any $\lambda\in \Lambda$ as an element of ${ T}^{^\star}$ by
 $\lambda(\dot T)=\lambda(\CV):=\{0\}$ and $\lambda(d):=d(\lambda)$ for any $d\in\CV^\dag$. Then clearly
 $t_{\lambda}=\lambda.$ If for $\a\in T^{^\star}$
  we define $\fg_\a$ in the usual manner, then
it is  easy to verify that for any $\lambda \in \Lambda$,
\begin{equation}\label{rootspacesslinfty}
\begin{array}{c}
  \fg_{\dot\alpha_{ij}+\lambda}=\CA^\lambda e_{ij},\quad (\dot\alpha_{ij}\neq 0),\vspace{2mm}\\
  \fg_\lambda = \; \mbox{the set of  diagonal matrices in}\;\CK \;\mbox{with enteries from} \;\CA^\lambda, {(\lam \neq 0),}\vspace{2mm}\\
{  \fg_0 = \; (\mbox{the set of diagonal matrices in}\;\CK
\;\mbox{with enteries from} \;\CA^0) \oplus \CV \oplus \CV^\dag.}

\end{array}
\end{equation}
  So $\fg=\bigoplus_{\dot\alpha \in \dot R, \lambda \in \Lambda} \fg_{\dot\alpha + \lambda}$.
  Therefore $(\fg,T)$ is a toral pair with root system
  \[
  R=\dot R+\Lam, \]
 and (IA1) holds for $\fg$.
{We next show that (IA2) holds.} Fix $\lam\in\Lam$ and choose an
invertible element $a \in \CA^\lambda,$ {then for $i\not= j$ we
have}
\begin{equation}\label{ex2 IA2}
\begin{array}{c}
 [a e_{ij}, a^{-1} e_{ji}]= e_{ii}-e_{jj} + \sum_{s\in J}([d_s,ae_{ij}],a^{-1}e_{ji})\lam_s \in T\vspace{2mm}\\

 [a(e_{ii}-e_{jj}),a^{-1}(e_{ii}-e_{jj})]
 =2\sum_{s\in J} d_s(\lambda)\ep(a,a^{-1}) \lam_s \in T,
\end{array}
\end{equation}
where {the first equality is always nonzero} and the second
{equality} is nonzero if $\lam\not=0$. Note that if $i\not=j$,
then $a^{\pm1}e_{ij}\in\fg_{{\dot\a}_{ij}\pm\lam}$, and
$a^{\pm1}(e_{ii}-e_{jj})\in\fg_{\pm\lam}$. So (IA2) holds.
Finally, since $\dot R$ is a \lfrs of type $\dot A_{K^\pm}$, (IA3)
holds by a similar argument as in the proof of
Theorem~\ref{gt0isiara}. Consequently, $(\fg,T)$ is an IARA.

{We further show that $\fg$ is division if and only if $\CA$ is
division graded. Using the fact that the elements of  $\dot T$ are
diagonal matrices with trace zero, it is not difficult to see that
if $\fg$ is division, then $\CA$ is division graded. Assume now
that $\CA$ is division graded. We must show that (IA2)$'$ holds
for $\fg$.} {Let $\lam \in \Lam$, $0 \neq a \in \CA^\lam$ and $i
\neq j \in K^\pm$, then
\[
[a e_{ij}, a^{-1} e_{ji}]= e_{ii}-e_{jj} + \sum_{s\in
I}([d_s,ae_{ij}],a^{-1}e_{ji})\lam_s \in
T\setminus\{0\}\vspace{2mm}
\]
as required. Also if $\sum_{i \in K^\pm_0} {a_i} e_{ii} \in
\fg_\lam,$ for a finite subset $K^\pm_0$ of $K^\pm$, where $0 \neq
a_i \in \CA^\lam$ and $\lam \neq 0$}, {then} {\begin{align*}
[\sum_{i \in K^\pm_0} a_ie_{ii}, \sum_{i \in K^\pm_0} a_i^{-1}
e_{ii} ] &= \sum_{s \in J} d_s(\lam)
                                        (\sum_{i \in K^\pm_0} \epsilon(a_i,a_i^{-1})) \lam_s\\
                                        &= \sum_{s \in J} d_s(\lam) |K^\pm_0| \lam_s \in
                                        T\setminus\{0\}
\end{align*}
as required.} Therefore $\fg$ is division if and only if $\CA$ is
division graded.
Indeed by \cite[Section 4.5]{nehersurvey}, $\fg$ is division if
and only if {$B$ } is division. So from now on, {\it we assume
that $B$ is division.}

There exists an involution $\bar{\;}$ (a self-inverting
anti-automorphism) on $\CA$ (see \cite[Section 2]{ag}) such that
$\bar{z_j}=z_j,$ for any $ j \in J$ and $\bar b=b$ for all $b\in
B$. By definition, it is clear that $\epsilon(\bar a,\bar
b)=\ep(a,b)$ for any $a,b\in \CA$. Using the involution
$\bar{\;},$ we can define an involution $\;^\ast$ on $\CK$ by
{$(ae_{ij})^\ast= \overline{a}e_{-j,-i}$}. It is straightforward
to see that the linear map $\sigma:\fg\lra\fg$ defined by
$$\sigma(x)=-x^\ast \;\mbox{for}\; x\in \CK\andd \sigma(x)=x \hbox{ for } x\in \CV\oplus\CV^\dag,
$$
is a Lie algebra automorphism.

We will show that $\sigma$ satisfies (A1)-(A5). Clearly
$\sigma^2(x)=x$ for any $x\in \fg$, thus $\sigma$ satisfies (A1)
with $m=2$. Also it is clear from definition that $\sigma$
satisfies (A2). In addition, observe that
\begin{eqnarray*}
  (\sigma(ae_{ij}),\sigma(be_{ks}))&=&((ae_{ij})^\ast,(be_{ks})^\ast)\\
                                   &=&(\bar ae_{-j,-i},\bar be_{-s,-k})\\
                                   &=&\delta_{jk}\delta_{is}\ep(\bar a,\bar b)\\
                                   &=&\delta_{jk}\delta_{is}\ep( a, b)\\
                                   &=&(ae_{ij},be_{ks}).
\end{eqnarray*}
  So (A3) holds for $\sigma$. Since $m=2$ is prime, instead of (A4) we will show that $\sigma$ satisfies the equivalent condition (A4)$''$
  (see Lemma \ref{newtemplem9}) namely, we show that for $0\not=\a\in R$, $\pi(\a)\not=0$.
  Recall from Section \ref{gradings} that since $\sg$ satisfies (A2),
  it induces an automorphism on $T^\star$,
  denoted again by $\sg$. We now note that $^\ast$ maps diagonal matrices to diagonal
  matrices and $\bar{\;}$ preserves homogeneous subspaces of $\CA$. Thus by (\ref{rootspacesslinfty})
  for any $\lambda \in \Lambda$, $\sigma(\fg_\lambda)=\fg_\lambda$, implying $\sigma(\lambda)=\lambda$.
  Consequently, $\pi(\lambda)=\lambda$ and so if $\lambda\neq 0, $ then so is $\pi(\lambda)$.
  On the other hand, we have $\sigma^{-1}=\sigma$ so for any $t\in T$ and $i\neq j \in K^\pm$,
\begin{eqnarray*}
  \pi(\dot\a_{ij})(t)&=& \frac{1}{2}(\dot\a_{ij}+\sigma(\dot\a_{ij})(t))\\
        &=& \frac{1}{2} (\dot\a_{ij}(t)+\dot\a_{ij}(\sg(t)))\\
        &=& \frac{1}{2} (\dot\a_{ij}(t+\sigma (t))).
\end{eqnarray*}
  Using this, we see that for $i \neq j \in K^\pm$ and  $t:=e_{ii}-e_{jj},$
\rred{\[
\pi(\dot\a_{ij})(t)= \left\{
\begin{array}{ll}
  1 & -j \neq i \;\mbox{and}\; i,j \neq 0\\
  \frac{1}{2} & -j \neq i , i=0 \;\mbox{or}\; j =0\\
  2 & -j=i.
\end{array}
\right.
\]}
   Consequently, $\pa\neq 0$ for any $0 \neq \alpha \in R${. I}n particular (A4)$''$ holds.

{We next show that (A5) holds.  Let $i \neq j \in K \cup \{0\}$,
then we have
\begin{eqnarray*}
 0 \neq e_{ii}-e_{-i-i} \in T^{0}\andd 0 \neq e_{ii}-e_{jj} + e_{-i-i} - e_{-j-j} \in T^{\bar 1}.
\end{eqnarray*}
In particular (A5) holds. Therefore $(\fg, T)$ and $\sigma$
satisfy all requirements of Theorem~\ref{Lrhofgisiara} and so we
can  construct a new IARA $(\hat\fg, \hat T)$.}

{Note that by (\ref{rootspacesslinfty}),
$$\fg_0=(\hbox{the set of diagonal matrices in $\CK$ with entries from
$\CA^0$})\oplus\CV\oplus\CV^0.$$ So, $\fg_0$ is abelian if and
only if $\CA^0 = B$ is abelian, indeed, if and only if $B$ is a
field.}

Now that we have a suitable automorphism on $\fg$, choosing a torsion-free abelian
group $\Lam'$, a group epimorphism $\rho:\Lam' \lra \Z_2$ and a predivision
$\Lam'$-graded commutative associative algebra $\CA'$, we can use Theorem~\ref{Lrhofgisiara}
to construct another IARA, $\hat\fg$ with a root system $\hat R$.

\rred{It is now interesting to have a discussion on the type of
$\hat\fg$. Note that we have
\[
R = \{ \epsilon_i  - \epsilon_j + \lam \mid i \neq j \in K^\pm, \lam \in \Lam \}.
\]
By definition of $\sg$ one can easily check that for any $i \in
K$, $\sg(\epsilon_i)=-\epsilon_{-i}$, and as we have already
seen $\sg(\lam)=\lam$ for any $\lam \in \Lam$. Therefore
\begin{align*}
\pi(R)  = &\; \{ \frac{1}{2}(\epsilon_i - \epsilon_j  +\epsilon_{-j} - \epsilon_{-i}) + \lam \mid i \neq j \in K^\pm, \lam \in \Lam\}\\
= &\;  \{ \pm \frac{1}{2}(\epsilon_i - \epsilon_{-i}) + \lam\mid i \in K, \lam \in \Lam\}\\
& \cup\; \{ \pm\frac{1}{2}((\epsilon_i - \epsilon_{-i}) \pm (\epsilon_j  -\epsilon_{-j} )) + \lam \mid i \neq j \in K, \lam \in \Lam\}\\
&\cup\; \{ \pm (\epsilon_i-\epsilon_{-i}) + \lam \mid i \in K, \lam \in \Lam \}
\end{align*}
This makes it clear that $\pi(R)$ is an affine reflection
system of type $BC$. But by (\ref{newtempeq11}), $\hat{R}$ and
$\pi(R)$ have the same type. Thus $\hat\fg$ is an IARA of type
$BC$.}
\end{ex}

\newcommand{\etalchar}[1]{$^{#1}$}
\providecommand{\bysame}{\leavevmode\hbox
to3em{\hrulefill}\thinspace}
\providecommand{\MR}{\relax\ifhmode\unskip\space\fi MR }
\providecommand{\MRhref}[2]{%
  \href{http://www.ams.org/mathscinet-getitem?mr=#1}{#2}
} \providecommand{\href}[2]{#2}

\end{document}